\documentclass[12pt]{article}
\usepackage{amsmath,amssymb}
\usepackage{color,graphicx}
\usepackage{ulem}

\newcommand{\bv}{\boldsymbol v}

\newcommand{\bz}{\boldsymbol z}

\newcommand{\bvar}{\boldsymbol \varphi}
\newcommand{\bU}{\boldsymbol U}
\newcommand{\bV}{\boldsymbol V}

\DeclareMathOperator*{\argmin}{argmin}

\newtheorem{Theorem}{Theorem}

\newtheorem{lema}{Lemma}

\newcounter{remark}
\def\theremark {\arabic{remark}}
\newenvironment{remark}{\refstepcounter{remark}\par\noindent{\bf Remark\ \theremark}\ }{\par}
\newtheorem{Proof}{Proof}

\newenvironment{proof}{\begin{Proof}\rm}{\hfill $\Box$ \end{Proof}}

\usepackage{hyperref}
\usepackage[all]{hypcap}

\title{Optimal bounds for POD approximations of infinite horizon control problems based on time derivatives}

\author{ Javier de Frutos\thanks{Instituto de Investigaci\'on en Matem\'aticas (IMUVA), Universidad
de Valladolid, Spain. Research supported
by Spanish MCIN/AEI 
under grants PID2019-104141GB-I00  and  PID2022-136550NB-I00 co-finanzed by FEDER (EU) funds and by Junta de Castilla y Le\'{o}n  under grant VA169P20 co-finanzed by FEDER (EU) funds
(frutos@mac.uva.es)}
\and Bosco Garc\'{\i}a-Archilla\thanks{Departamento de Matem\'atica Aplicada
II, Universidad de Sevilla, Sevilla, Spain. Research is supported by
Spanish  MCIN/AEI under grants PID2019-104141GB-I00 and PID2022-136550NB-I00 
co-finanzed by FEDER (EU) funds (bosco@esi.us.es)}
  \and Julia Novo\thanks{Departamento de
Matem\'aticas, Universidad Aut\'onoma de Madrid, Spain.  Research supported
Spanish  MCIN/AEI under grants PID2019-104141GB-I00 and PID2022-136550NB-I00 co-finanzed by FEDER (EU) funds and by Junta de Castilla y Le\'{o}n  under grant VA169P20 co-finanzed by FEDER (EU) funds
 (julia.novo@uam.es)}}
\date{\today}
\begin{document}
\maketitle
\begin{abstract}
In this paper we consider the numerical approximation of infinite horizon  problems via the dynamic programming approach.
The value function of the problem solves a Hamilton-Jacobi-Bellman (HJB) equation that is approximated by a fully discrete method. 
It is known that the numerical problem is difficult to handle by the so called curse of dimensionality. To mitigate this issue
we apply a reduction of the order by means of a new proper orthogonal decomposition (POD) method based on time derivatives. 
We carry out the error analysis of the method 
using recently proved optimal bounds for the
fully discrete approximations. Moreover, the use of snapshots based on time derivatives allows us to bound some terms of the
error that could not be bounded in a standard POD approach. Some numerical experiments show the good performance of the method
in practice.
\end{abstract}
\bigskip


\bigskip

{\bf Data Availability Statement.} The authors declare that the data supporting the findings of this study can be obtained by reproducing the computations as described in the section on numerical experiments.

\section{Introduction}
In this paper we consider the numerical approximation of optimal control problems. The subject is  of importance for many applications such as aerospace engineering, chemical processing and resource economics, among others.

The value function of an optimal control problem is obtained in terms of a first-order nonlinear Hamilton-Jacobi-Bellman (HJB) partial differential equation.
A bottleneck in the computation of the value function comes from the need to approach a nonlinear partial differential equation in dimension $n$, which is a challenging problem in high dimensions.

Several methods have been studied in the literature trying to mitigate the so called curse of dimensionality although it is still
a difficult task. As stated in \cite{eigel_et_al}, the relevance of efficient numerical methods can be seen by the fact that methods solving the HJB equation are rarely used in practice due to the necessary computational effort. We mention some related references that 
are not intended to be a complete list. In  \cite{falcone_dd}, 
 a domain decomposition technique is considered. In \cite{tonon}  semi-Lagrangian methods are studied.
The authors in \cite{Luo_et_al} apply data-based approximate policy iteration methods. 
A procedure for the numerical approximation of high-dimensional HJB equations associated to optimal feedback control problems for semilinear parabolic equations is proposed in \cite{Kalise_et_al}. In \cite{Dolgov_et_al} a tensor decomposition approach is presented. In \cite{eigel_et_al} an approach based on low-rank tensor train decompositions is applied.
\textcolor{red}{Methods using sparse grids for HJB equations are presented in \cite{boba2}. The solution of HJB equations
on a tree structure was presented in 
\cite{Alla_fal_sa}. The author of \cite{maceneay1}, \cite{maceneay2} discusses an approach to certain nonlinear HJB PDEs which is not subject to the curse of dimensionality. The approach utilizes the max-plus algebra.
In \cite{dante_etal} a data-driven approach based on the knowledge of the value function and its gradient on sample points is developed.
The authors of \cite{Alla_et_all} present a new approach where the value function is computed using radial basis functions. Expanded literature on the control of partial differential equations using dynamic programming approach can be
found in the last two references.}

In the present paper we concentrate on reduced order models based on proper orthogonal decomposition (POD) methods. Our work is related to \cite{Alla_Falcone_Volkwein}. In this reference the authors propose two different ways to apply POD methods in the numerical approximation of the fully-discrete value function. In the first approach, the authors choose a set of nodes in the original domain $\Omega\subset {\Bbb R}^n$ and project then onto a reduced space $\Omega^r \subset {\Bbb R}^r$ with $r<n$ to get a new set of nodes. The problem in this procedure is that it produces a nonuniform grid in which the mesh diameter cannot be predicted a priori. Consequently, the method
is not suitable to implement in practice. Furthermore, although this is not reflected in the error bounds in  \cite{Alla_Falcone_Volkwein}, the error
also depends on the interpolation properties of the a priori unknown reduced mesh in $\Omega^r \subset {\Bbb R}^r$. In the second approach, the authors use a uniform mesh over the reduced space $\Omega^r$. This second method can be implemented in practice (the numerical
experiments in \cite{Alla_Falcone_Volkwein} are carried out with this method). However, as the authors state, the computation of an upper
bound for the error in this case is much more involved and the error bound proved in \cite{Alla_Falcone_Volkwein} has some drawbacks see \cite[Remark 4.7, Remark 4.9]{Alla_Falcone_Volkwein}. 

Recently in \cite{Javier_yo}, a new error analysis is introduced in which 
a bound of size  $O(h+k)$ is obtained for the fully discrete approximations to infinite horizon problems via the dynamic programming approach. In this error bound, $h$ is the time step while $k$ is the spatial mesh diameter. This error bound improves existing
results in the literature, where only $O(k/h)$ error bounds are proved, see \cite{falcone1}, \cite{falcone11}. 

To bound the error in the first method in  \cite{Alla_Falcone_Volkwein} the authors follow the technique in  \cite[Corollary 2.4]{falcone1}, \cite[Theorem 1.3]{falcone11} obtaining
a bound for the error of size $O(k/h)$. For the second method in \cite{Alla_Falcone_Volkwein}, the factor $1/h$ also multiplies all  the terms on the right-hand side of the a priori error bound. 

In this paper we present a new approach, similar to the second method
in \cite{Alla_Falcone_Volkwein}, but with snapshots based on the value at different times of the time derivative of the state of the controlled nonlinear dynamical system,  instead of values of the state
at different times. This new approach is inspired in the recent results in \cite{Bosco_Volker_yo} where the authors prove that the use of snapshots based on time derivatives has the advantage of providing pointwise estimates for the error between a function and its projection onto the POD space. \textcolor{red}{The idea of using snapshots approaching the time derivatives is not new, although
most of the references in the literature employ first order difference quotients (DQs) (i.e. first order divided finite differences) instead of Galerkin time derivatives, as
in \cite{Bosco_Volker_yo} and the present paper. In \cite{Ku-Vol} the set of snapshots (at different times) is increased with DQs to carry out the error analysis for the  case in which projections respect to the $H_0^1$ norm are considered. In a more recent paper, \cite{Koc_et_al},
the authors show that the use of DQs has the added property of allowing to prove pointwise estimates in time. In a later paper, \cite{locke-singler}, the authors prove that one does not need to double the set of snapshots with values at different times plus DQs since only DQs plus a single initial value are enough to get pointwise estimates. This is a very interesting result because one can work with the same number of snapshots as in the standard case (the one with only values of the states at different times). In \cite{Bosco_Volker_yo}, the authors prove that this is also the case with time derivatives. A set of snapshots based on time
derivatives plus the snapshot at the initial time (or the mean value of the states) is able to provide pointwise in time
error estimates. This is the idea we apply in the present paper.}
Moreover, we carry out a different error analysis based on the recent results obtained in \cite{Javier_yo} that allow us to get sharper error bounds free of $1/h$ factors. This is in agreement with the numerical investigations in the literature where the $1/h$ behaviour in the error bounds of fully discrete methods has never been observed. Also, the use of snapshots based on time derivatives allows us to give a bound for some  terms 
that could not be bounded with the standard approach. Both facts, the new technique
used to bound the error that follows ideas in \cite{Javier_yo} together with the use of snapshots based on time derivatives, are the
key ingredients to get error bounds for the
new method that are optimal in terms of the time step $h$ and the mesh diameter of the reduced space $k_r$. As usual, our error bounds for 
the POD method depend also on the size of the tail of eigenvalues in the singular value decomposition.

The outline of the paper is as follows. In Section 2 we state the model problem and some preliminary results. In Section 3 we introduce the
POD approximation and carry out the error analysis of the method. Finally, in Section 4 we show some numerical experiments in which we implement the method we propose in the paper. In the experiments of Section 4 we choose the same numerical tests as in \cite{Alla_Falcone_Volkwein} to compare our results with those in this related
reference.
The method introduced in the present paper seems to produce better results than those shown in~\cite{Alla_Falcone_Volkwein}. We finish the paper with some conclusions.


\section{Model problem and standard numerical approximation}
In the sequel, $\|\cdot\|$ denotes any norm  associated to an inner product and~$\|\cdot\|_\infty$  denotes the maximum norm for vectors in~${\mathbb R}^n$, $n\ge 1$. We will also denote by  $\|\cdot\|_2$ the standard euclidean norm. \textcolor{red}{In particular, in the numerical experiments, we use a weighted norm $\|\cdot\|$ slightly different from the standard euclidean norm $\|\cdot\|_2$}.

For a nonlinear mapping
$$
f:{\Bbb R}^n\times {\Bbb R}^m\rightarrow {\Bbb R}^n,
$$
and a given initial condition $y_0 \in {\Bbb R}^n$ let us consider the controlled nonlinear dynamical system
\begin{equation}\label{din_sis}
\dot y(t)=f(y(t),u(t))\in {\Bbb R}^n,\quad t>0,\quad y(0)=y_0\in {\Bbb R}^n,
\end{equation}
together with the infinite horizon cost functional
\begin{equation}\label{fun_cos}
J(y,u)=\int_0^\infty g(y(t),u(t))e^{-\lambda t}~dt.
\end{equation}
In \eqref{fun_cos} $\lambda>0$ is a given weighting parameter and
$$
g:{\Bbb R}^n\times {\Bbb R}^m\rightarrow {\Bbb R}.
$$
The set of admissible controls is
$$
{\Bbb U}_{\rm ad}=\left\{u\in {\Bbb U}\mid u(t)\in U_{\rm ad} \ {\rm for}\ {\rm almost}\ {\rm all}\ t\ge 0\right\},
$$
where ${\Bbb U}=L^2(0,\infty;{\Bbb R}^m)$ and $U_{\rm ad}\subset {\Bbb R}^m$ is a compact convex subset.

As in \cite[Assumption 2.1]{Alla_Falcone_Volkwein} we assume the following hypotheses:
\begin{itemize}
\item The right-hand side $f$ in \eqref{din_sis} is continuous and globally Lipschitz-continuous in both the first
and second arguments; i.e., there exists
a constant $L_f>0$ satisfying
\begin{eqnarray}\label{lip_f}
\|f(y,u)-f(\tilde y,u)\|&\le& L_f\|y-\tilde y\|,\quad \forall y,\tilde y\in{\Bbb R}^n, u\in U_{\rm ad},\\
\|f(y,u)-f(y,\tilde u)\|&\le& L_f \|u-\tilde u\|,\quad \forall u,\tilde u\in U_{\rm ad}, y\in {\Bbb R}^n.
\label{lip_f2}
\end{eqnarray}
\item The right-hand side $f$ in \eqref{din_sis}  satisfies that there exists a constant $M_f>0$ such that the following bound holds
\begin{equation}\label{infty_f}
\|f(y,u)\|_\infty\le M_f, \quad \forall y\in \overline \Omega\subset {\Bbb R}^n, u\in U_{\rm ad},
\end{equation}
where $\overline \Omega$ is a bounded polyhedron
 such that for sufficiently small $h>0$ {the following inward pointing condition on the dynamics} holds
\begin{equation}\label{invariance}
y+hf(y,u)\in \overline \Omega,\quad \forall y \in \overline \Omega, u\in U_{\rm ad}.
\end{equation}
\item The running cost $g$ is continuous and globally Lipschitz-continuous in both the first and second arguments; i.e., there exists a constant $L_g>0$
satisfying
\begin{eqnarray}\label{lip_g}
|g(y,u)-g(\tilde y,u)|&\le& L_g\|y-\tilde y\|,\quad \forall y,\tilde y\in{\Bbb R}^n, u\in U_{\rm ad},\\
|g(y,u)-g(y,\tilde u)|&\le& L_g \|u-\tilde u\|,\quad \forall u,\tilde u\in U_{\rm ad}, y \in {\Bbb R}^n.\label{lip_g2}
\end{eqnarray}
\item Moreover, there exists a constant $M_g>0$ such that
\begin{equation}\label{cota_g}
|g(y,u)|\le M_g,\quad \forall (y,u)\in \overline \Omega\times U_{\rm ad}.
\end{equation}
\end{itemize}
From the assumptions made on $f$ there exists a unique solution of \eqref{din_sis} $y=y(y_0,u)$ {defined on $[0,\infty)$} for every admissible control
$u\in {\Bbb U}_{\rm ad}$ and for every initial condition $y_0\in {\Bbb R}^n$, see \cite[Chapter 3]{Bardi}. We define the reduced cost functional as follows:
\begin{eqnarray}\label{eq:funcional}
\hat J(y_0,u)=J(y(y_0,u),u),\quad \forall u\in {\Bbb U}_{\rm ad},\quad y_0\in {\Bbb R}^n,
\end{eqnarray}
where $y(y_0,u)$ solves \eqref{din_sis}. Then, the optimal control can be formulated as follows: for given $y_0\in {\Bbb R}^n$
we consider
$$
\min_{u\in {\Bbb U}_{\rm ad}} \hat J(y_0,u).
$$
The value function of the problem is defined as $v:{\Bbb R}^n\rightarrow {\Bbb R}$ as follows:
\begin{equation}\label{eq_v}
v(y)=\inf\left\{\hat J(y,u)\mid u\in {\Bbb U}_{\rm ad}\right\},\quad y\in {\Bbb R}^n.
\end{equation}
This function gives the best value for every initial condition, given the set of admissible controls $U_{\rm ad}$. It is characterized as the viscosity solution of the HJB equation corresponding to the infinite horizon {optimal control problem:}
\begin{equation}\label{HJB}
\lambda v(y)+\sup_{u\in U_{\rm ad}}\left\{-f(y,u)\cdot \nabla v(y)-g(y,u)\right\}=0,\quad y\in {\Bbb R}^n.
\end{equation}
{The solution of (\ref{HJB})} is unique for sufficiently large $\lambda$, $\lambda>\max(L_g,L_f)$, \cite{Bardi}.

Let
us consider first a time discretization where $h$ is a strictly positive step size.
{We consider the following} semidiscrete scheme for \eqref{HJB}:
\begin{equation}\label{discrete_HJB}
v_h(y)=\min_{u\in U_{\rm ad}}\left\{(1-\lambda h)v_h(y+hf(y,u))+h g(y,u)\right\},\quad y\in {\Bbb R}^n.
\end{equation}
\textcolor{red}{As it is well-known equation \eqref{discrete_HJB} represents a numerical approximation related to the HJB equation
\eqref{HJB} (see Remark 7}). 
{The following  convergence result  for the semidiscrete approximation \cite[Theorem 2.3]{falcone1}  requires that
for $(y,\tilde y,u)\in {\Bbb R}^n\times {\Bbb R}^n\times U_{\rm ad}$
\begin{eqnarray}
 \|f(y+\tilde y,u)-2f(y,u)+f(y-\tilde y,u)\|&\le& C_f\|\tilde y\|^2,\label{semi_con_f}\\
 \|g(y+\tilde y,u)-2g(y,u)+g(y-\tilde y,u)\|&\le& C_g\|\tilde y\|^2.\label{semi_con_g}
 \end{eqnarray}
 }

\begin{Theorem}\label{th_semi} Let assumptions \eqref{lip_f}, \eqref{infty_f}, \eqref{invariance}, \eqref{lip_g}, \eqref{cota_g}, \eqref{semi_con_f} and \eqref{semi_con_g} hold and let $\lambda>\max(2L_g,L_f)$. Let $v$ and $v_h$ be the solutions
of \eqref{HJB} and \eqref{discrete_HJB}, respectively. Then, there exists a constant $C\ge 0$, that can be bounded explicitly, such that the
following bound holds
\begin{equation}\label{cota_semi}
\sup_{y\in {\Bbb R}^n}|v(y)-v_h(y)|\le Ch,\quad h\in [0,1/ \lambda).
\end{equation}
\end{Theorem}
As in \cite{Alla_Falcone_Volkwein} let us suppose that there exists a bounded polyhedron $\Omega\subset {\Bbb R}^n$ such that
for $h>0$ small enough \eqref{invariance} holds.
We consider a fully-discrete approximation to \eqref{HJB}.
 Let $\left\{S_j\right\}_{j=1}^{m_s}$ be a family of simplices which defines a regular triangulation of $\Omega$
$$
\overline \Omega=\bigcup_{j=1}^{m_s} S_j,\quad k=\max_{1\le j\le m_s}({\rm diam} \ S_j).
$$
We assume we have $n_s$ vertices/nodes $\color{red}\hat y^1,\ldots,\hat y^{n_s}$ in the triangulation. Let $V^k$ be the space of piecewise affine functions from $\overline \Omega$ to ${\Bbb R}$ which are continuous in $\overline \Omega$ having constant gradients in the interior of any simplex $S_j$ of the triangulation. Then, a fully discrete scheme for the HJB equations is given by
\begin{equation}\label{fully_discrete}
v_{h,k}(\hat y^i)=\min_{u\in {U}_{\rm ad}}\left\{(1-\lambda h)v_{h,k}(\hat y^i+hf(\hat y^i,u))+hg(\hat y^i,u)\right\},
\end{equation}
for any vertex $\hat y^i\in \overline \Omega$. There exists a unique
solution of \eqref{fully_discrete} in the space $V^k$, see \cite[Theorem 1.1, Appendix A]{Bardi}.

For the fully discrete method if we assume that the controls are Lipschitz-continuous; i.e., there exists a positive constant $L_u>0$ such that
\begin{eqnarray}\label{lip_u}
\|u(t)-u(s)\|_2\le L_u |t-s|,
\end{eqnarray}
then first order of convergence both in time and space is proved in \cite[Theorem 6]{Javier_yo}.
\begin{Theorem}\label{th_6}
Assume conditions  \eqref{lip_f}, \eqref{lip_f2}, \eqref{infty_f}, \eqref{lip_g}, \eqref{lip_g2}, \eqref{cota_g} and \eqref{lip_u} hold.
Assume $\lambda>\overline L$ with $\overline L=CnL_f$. Then, for $0\le h\le 1/(2\lambda)$ there exist positive constants ${C_1=
C_1(\lambda,M_f,M_g,L_f,L_g)}$ and  ${C_2=
C_2(\lambda,L_f,L_g,L_u)}$ such that
$$
|v(y)-v_{h,k}(y)|\le C_1(h+k)+{C_2h},\quad y\in \overline \Omega.
$$
\end{Theorem}
Condition \eqref{lip_u} can be weakened and one can still get convergence as proved in \cite[Theorem 7]{Javier_yo}.
Assume \textcolor{red}{the following convexity assumption introduced in \cite[(A4)]{boba_etal} and denoted by (CA) as in \cite{boba_etal}, \cite{Javier_yo}}
\begin{itemize}
\item (CA) For every $y\in {\Bbb R}^n$,
\begin{eqnarray*}
\left\{f(y,u), g(y,u),\quad  u\in U_{\rm ad}\right\}
\end{eqnarray*}
is a convex subset of ${\Bbb R}^{n+1}$.
\end{itemize}
\begin{Theorem}\label{th_4_cons}
Assume conditions  \eqref{lip_f}, \eqref{lip_f2}, \eqref{infty_f}, \eqref{lip_g}, \eqref{lip_g2}, \eqref{cota_g} and {\rm(CA)}  hold.
Assume $\lambda>\overline L$ with $\overline L$ defined as in Theorem \ref{th_6}. Then, for $0\le h\le 1/(2\lambda)$ there exist positive constants $C_1=
C_1(\lambda,M_f,M_g,L_f,L_g)$ and  $C_2=
C_2(\lambda,M_f,M_g,L_f,L_g)$ such that for $y\in \overline \Omega$
\begin{equation}\label{cota_buena_cons}
|v(y)-v_{h,k}(y)|\le C_1(h+k)+C_2 \frac{1}{(1+\beta)^2\lambda^2}(\log(h))^2h^{\frac{1}{1+\beta}},\quad \beta=\frac{\sqrt{n}L_f}{\lambda}.
\end{equation}
\end{Theorem}
Let us observe that since $\beta$ is smaller than 1, by weakening the regularity requirements we loose at most half an order in the rate of convergence in time
of the method up to a logarithmic term.
\section{POD approximation of the optimal control problem based on time derivatives}
In this section we present a new approach, similar to the second method
in \cite{Alla_Falcone_Volkwein}, but with snapshots based on time derivatives at different times. We also perform a completely different error analysis to the one appearing in \cite{Alla_Falcone_Volkwein}, inspired in the results in \cite{Javier_yo} and \cite{Bosco_Volker_yo}.
\subsection{POD approximation based on time derivatives}\label{Se:3.1}
For  $p\in \Bbb N$ let us choose different pairs $\left\{(u^\nu,y_0^\nu)\right\}_{\nu=1}^p$ in ${\Bbb U}\times \overline \Omega$.
\textcolor{red}{Since ${\Bbb U}=L^2(0,\infty;{\Bbb R}^m)$ 
the controls do not need to be constants, as those taken in the numerical 
experiments}.
By $y^\nu=y(u^\nu;y_0^\nu)$, $\nu=1,\ldots,p$, we denote the solutions of \eqref{din_sis} corresponding to those chosen initial conditions
and controls.

Let us fix $T>0$ and $M>0$ and take $\Delta t=T/M$ and $t_j=j\Delta t$, $j=0,\ldots, M$. For $N=M+1$ we define the following space
\[
{\color{red}{\cal \bV}}=\mbox{span}\left\{z_1^\nu,z_2^\nu,\ldots,z_N^\nu\right\}_{\nu=1}^p,
\]
with
\begin{eqnarray*}
z_1^\nu&=&\sqrt{N}\overline y^\nu,\quad \overline y^\nu=\frac{1}{N}\sum_{j=0}^My^\nu(t_j)\\
z_j^\nu&=&\tau y_{t}^\nu(t_{j-1}),\ j=2,\ldots,N,
\end{eqnarray*}
so that
\[
{\cal \bV}= 
\mbox{span}\left\{\sqrt{N}\overline y^\nu,\tau y_{t}^\nu(t_{1}),\ldots,\tau y_{t}^\nu(t_{N})\right\}_{\nu=1}^p,
\]
where the factor $\tau$ in front of the temporal derivatives is a time scale and it makes the 
snapshots dimensionally correct. \textcolor{red}{In the numerical experiments we take $\tau=1$}.
The correlation  matrix corresponding to the snapshots is given by ${K}=((k_{i,j}))\in {\mathbb R}^{pN\times pN}$,
with the entries 
\[
k_{i,j}=\frac{1}{pN}\left(z_k^{i},z_l^{j}\right), \quad k,l=1,\ldots,N, \quad i,j=1,\ldots, p,
\]
and where here, and in the sequel, $(\cdot,\cdot)$ denotes the inner product in ${\Bbb R}^n$ to which the norm  $||\cdot||$
is associated.
Let us denote for simplicity
\[
{\cal \bV}=\mbox{span}\left\{w_1,w_2,\ldots,w_{pN}\right\}:=\left\{z_1^1,\ldots z_N^1,\ldots,z_1^p,\ldots,z_N^p\right\}.
\]
Following \cite{Ku-Vol}, we denote by
$ \lambda_1\ge  \lambda_2,\ldots\ge \lambda_{d}>0$ the positive eigenvalues of {$K$} and by
$\bv_1,\ldots,\bv_{d}\in {\mathbb R}^{pN}$ its associated eigenvectors of euclidean norm $1$.  
Then, the (orthonormal) POD basis functions of $\cal \bV$ are given by
\begin{equation}\label{lachi}
\varphi_k=\frac{1}{\sqrt{pN}}\frac{1}{\sqrt{\lambda_k}}\sum_{j=1}^{pN} v_k^j w_j,\ k=1,\ldots,d,
\end{equation}
where $v_k^j$ is the $j$-th component of the eigenvector $\bv_k$.
The following error estimate is known from \cite[Proposition~1]{Ku-Vol}
\begin{eqnarray}\label{cota_ku}
\frac{1}{pN}\sum_{j=1}^{pN}\left\|w_j-\sum_{k=1}^r(w_j,\varphi_k)\varphi_k\right\|^2=\sum_{k=r+1}^{d}\lambda_k,
\end{eqnarray}
from which one can deduce for $\nu=1,\ldots,p$
\begin{equation}
\left\|\overline y^\nu-\sum_{k=1}^r(\overline y^\nu,\varphi_k)\varphi_k\right\|^2+\frac{\tau^2}{M+1}\sum_{j=1}^M\left\|y_{t}^{\nu}(t_j)-\sum_{k=1}^r(y_{t}^{\nu}(t_j),\varphi_k)\varphi_k\right\|^2\le p\sum_{k=r+1}^{d}\lambda_k.\label{eq:cota_pod_deriv}
\end{equation}
In the sequel, we will denote by
\begin{eqnarray}\label{Ur}
{\color{red}{\cal \bV}^r}= \mbox{span}\{\varphi_1,\varphi_2,\ldots,\varphi_r\},\quad 1\le r\le d,
\end{eqnarray}
and by $P^r\ : \  {\Bbb R}^n  \to {\cal \bV}^r$,  the orthogonal projection onto ${\cal \bV}^r$. 
Then \eqref{cota_ku} can be written as
\[
\frac{1}{pN}\sum_{j=1}^{pN}\left\|w_j-P^r w_j\right\|^2=\sum_{k=r+1}^{d}\lambda_k.
\]
The following lemma is proved in \cite[Lemma 3.2]{Bosco_Volker_yo}.
\begin{lema}\label{le:our_etal_mean}
Let $T>0$, $\Delta t=T/M$, $t^n=n\Delta t$, $n=0,1,\ldots M$, let $X$ be a {Banach} space, $\bz\in H^2(0,T;X)$. Then, the following estimate holds
\begin{equation}\label{eq:zetast_mean}
\max_{0\le k\le N }\|\bz^k\|_X^2 \le  {3}\|\overline\bz\|_X^2+\frac{12 T^2}{M}\sum_{n=1}^M \| \bz_t^n\|_X^2+\frac{16T}{3}(\Delta t)^2\int_0^T\|\bz_{tt}(s)\|_X^2\ ds,
\end{equation}
where
$
\overline \bz=\frac{1}{M+1}\sum_{j=0}^M\bz^j$.
\end{lema}
Using Lemma \ref{le:our_etal_mean} we can prove pointwise estimates for the projections onto $ {\cal \bV}^r$.
\begin{lema} \label{le:poin} The following bounds hold for $\nu=1,\ldots,p$
\begin{equation}\label{eq:bound_2nd_term}
\max_{0\le j\le M }\|y^\nu(t_j)-P^r y^\nu(t_j)\|^2
\le \left(3+24\frac{T^2}{\tau^2}\right)p\sum_{k={r+1}}^{d}\lambda_k
+\frac{16T}{3}(\Delta t)^2 \int_0^T\|y^{\nu}_{tt}(s)\|^2\ ds.
\end{equation}
\end{lema}

\begin{proof}
We argue as in \cite[Lemma 3.4]{Bosco_Volker_yo}.
Taking $\bz=y^\nu(t_j)-P^r y^\nu(t_j)$ in \eqref{eq:zetast_mean} and applying \eqref{eq:cota_pod_deriv} ({taking into account that
$(M+1)/M\le2$}) yields
\begin{eqnarray*}
\max_{0\le n\le M }\|y^\nu(t_j)-P^r y^\nu(t_j)\|^2&\le& \left(3+24\frac{T^2}{\tau^2}\right)p\sum_{k={r+1}}^{d}\lambda_k\nonumber\\
&&{}+\frac{16T}{3}(\Delta t)^2\int_0^T\|y^\nu_{tt}(s)-P^r y^\nu_{tt}(s)\|^2\ ds.
\end{eqnarray*}
Now, since $P^r$ is an orthogonal projection, we have $\|y^\nu_{tt}(s)-P^r y^\nu_{tt}(s)\|^2\le \|y^\nu_{tt}(s)\|^2$ and the proof is finished
\end{proof}
\subsection{The POD control problem}
To mitigate the curse of dimensionality, the idea of the POD method is to work on a space of dimension $r$ with $r<n$. To start
we need to introduce some notation. \textcolor{red}{We use a slightly different notation from the one used in \cite{Alla_Falcone_Volkwein}. In particular, as stated below, $P_c^r$. is always used to denote coefficients and $\bvar$
is always used for the linear combination based on the POD basis functions of the reduced order space. More precisely:}

For any $y\in \overline \Omega\subset {\Bbb R}^n$ let us denote by $P_c^r y\in {\Bbb R}^r$ the coefficients
of the projection of $y$ onto ${\cal \bV}^r$
\begin{eqnarray}\label{defi1}
P_c^r y=\left\{(y,\varphi_k)\right\}_{k=1}^r.
\end{eqnarray}
For any $y^r\in {\Bbb R}^r$ let us denote by $\bvar y^r\in {\Bbb R}^n$ the vector  whose coefficients in the
POD basis are the components of $y^r$, i.e., 
\begin{eqnarray}\label{defi2}
\bvar y^r=\sum_{j=1}^r y^r_j \varphi_j,
\end{eqnarray}
where $y^r_j$ is the $j$ component of the vector $y^r$.

For $f$ and $g$ in \eqref{din_sis}, \eqref{fun_cos} and $(y^r,u)\in {\Bbb R}^r\times U_{\rm ad}$ we define
\begin{eqnarray}\label{frgr}
f^r(y^r,u)&=&P^r_cf(\bvar y^r,u)\in {\Bbb R}^r,\\
g^r(y^r,u)&=&g(\bvar y^r,u)\in {\Bbb R}.\nonumber
\end{eqnarray}
To have an inward pointing condition on the dynamics in the reduced space, analogous to \eqref{invariance},
following \cite[Section 4.2]{Alla_Falcone_Volkwein}, we assume that there exists a bounded polyhedron $\overline\Omega^r\subset {\Bbb R}^r$
satisfying
\begin{equation}\label{con_alla}
P_c^r y\in \Omega^r,\quad \forall y\in \overline \Omega.
\end{equation}
The following lemma proves that the inward pointing condition for $\Omega^r$ follows from \eqref{con_alla}.
\begin{lema}\label{lema3}
Condition \eqref{con_alla} implies that 
$$
y^r+hf^r(y^r,u)\in \overline \Omega^r,\quad y^r=P_c^r y, \ y\in \overline \Omega,
$$
provided the step size $h$ or $\|P^r y-y\|$ is sufficiently small. 
\end{lema}
\begin{proof}
We follow \cite[Remark 4.5]{Alla_Falcone_Volkwein} for the proof. 
We first observe that
$$
y^r+hf^r(y^r,u)=P_c^r y+h P^r_c f(\bvar y^r,u).
$$
Adding and subtracting $hP^r_c f(y,u)$ we get
\begin{eqnarray}\label{eq_dentro_dom}
y^r+hf^r(y^r,u)=P_c^r(y+hf(y,u))+hP_c^r(f(\bvar y^r,u)-f(y,u)).
\end{eqnarray}
Applying condition \eqref{invariance} $y+hf(y,u)\in\overline \Omega$ and applying \eqref{con_alla} 
the first term on the right-hand side of \eqref{eq_dentro_dom} verifies $P_c^r(y+hf(y,u))\in \Omega^r$. Then, we only need
to show that the second term on the right-hand side of \eqref{eq_dentro_dom} is small enough for $h$ or $\|P^r y-y\|$ sufficiently small.

Let us denote by $z=f(\bvar y^r,u)-f(y,u)\in {\Bbb R}^n$. Since $P_c^r z=\left\{(z,\varphi_k)\right\}_{k=1}^r\in {\Bbb R}^r$, taking into account that the functions $\varphi_k$ define an orthonormal basis and
that $P^r$ is a projection,
we have $$\|P_c^r z\|_2=\| P^rz\| \le \| z\|.$$ 
Applying the above inequality together with \eqref{lip_f}, we get
\begin{eqnarray}\label{normas}
\|P_c^r z\|_2^2&=&\|P_c^r(f(\bvar y^r,u)-f(y,u))\|_2^2\le \|f(\bvar y^r,u)-f(y,u)\|^2
\nonumber\\
&\le&  L_f^2 \|\bvar y^r-y\|^2=L_f^2 \|P^r y-y\|^2,
\end{eqnarray}
so that the proof is concluded.
\end{proof}
We can now define the reduced order problem we solve in practice. For $f^r$ and $g^r$ defined in \eqref{frgr} and a given initial condition $y_0^r \in {\Bbb R}^r$ let us consider the controlled nonlinear dynamical system
\begin{equation}\label{din_sis_pos}
\dot y^r(t)=f^r(y^r(t),u(t))\in {\Bbb R}^r,\quad t>0,\quad y^r(0)=y_0^r\in {\Bbb R}^r,
\end{equation}
together with the infinite horizon cost functional
\begin{equation}\label{fun_cos_pod}
J^r(y^r,u)=\int_0^\infty g^r(y^r(t),u(t))e^{-\lambda t}~dt.
\end{equation}
As in \eqref{eq:funcional}, we define the reduced cost functional
\begin{eqnarray}\label{eq:funcional_pod}
\hat J^r(y_0^r,u)=J^r(y^r(y_0^r,u),u),\quad \forall u\in {\Bbb U}_{\rm ad},\quad y_0^r\in {\Bbb R}^r,
\end{eqnarray}
where $y^r(y_0^r,u)$ solves \eqref{din_sis_pos}. Then, the POD optimal control can be formulated as follows: for given $y_0^r\in {\Bbb R}^r$
we consider
$$
\min_{u\in {\Bbb U}_{\rm ad}} \hat J^r(y_0^r,u).
$$
The value function of the problem $v^r:{\Bbb R}^r\rightarrow {\Bbb R}$ is defined as follows:
\begin{equation}\label{eq_v_pod}
v^r(y^r)=\inf\left\{\hat J(y^r,u)\mid u\in {\Bbb U}_{\rm ad}\right\},\quad y^r\in {\Bbb R}^r.
\end{equation}
\begin{remark}\label{re_fr_gr}
It is easy to check that the regularity assumptions for $f^r$ and $g^r$ analogous to those for $f$ and $g$, \eqref{lip_f}, \eqref{lip_f2}, \eqref{infty_f}, \eqref{lip_g}, \eqref{lip_g2} and \eqref{cota_g}, hold from the definition of $f^r$ and $g^r$ 
and the properties being true for $f$ and $g$.  
\end{remark}
\bigskip

To get in practice a fully discrete approximation in the reduced space let us define $\left\{S_j^r\right\}_{j=1}^{m_s^r}$ a family of simplices which defines a regular triangulation of $\Omega^r$. We assume
we have $n_s$ vertices/nodes in the triangulation $\color{red}\hat y^{1}_{r},\ldots,\hat y^{n_s}_{r}\in \overline \Omega^r$ and 
$$
\overline \Omega^r=\bigcup_{j=1}^{m_s^r} S_j^r,\quad k_r=\max_{1\le j\le m_s^r}({\rm diam} \ S_j^r).
$$
Let $V^{k_r}$ be the space of piecewise affine functions from $\overline \Omega^r$ to ${\Bbb R}$ which are continuous in $\overline \Omega^r$ having constant gradients in the interior of any simplex $S_j^r$ of the triangulation.
As in \cite[(4.15)]{Alla_Falcone_Volkwein} we introduce the following POD fully discrete scheme for the HJB equations
\begin{equation}\label{fully_discrete_pod}
v_{h,k}^r(\hat y^{i}_{r})=\min_{u\in {U}_{\rm ad}}\left\{(1-\lambda h)v_{h,k}^r(\hat y^{i}_{r}+hf^r(\hat y^{i}_{r},u))+hg^r(\hat y^{i}_{r},u)\right\},\ i=1,\ldots,n_s,
\end{equation}
for any vertex $\hat y^{i}_{r}\in \overline \Omega^r$. As in \eqref{fully_discrete}, there exists a unique
solution of \eqref{fully_discrete_pod} in the space $V^{k_r}$ defined by its nodal values \eqref{fully_discrete_pod}, see \cite[Theorem 1.1, Appendix A]{Bardi}.

The key point to carry out the error analysis is that \eqref{fully_discrete_pod} is the fully discrete approximation to the continuous problem with 
value function defined in \eqref{eq_v_pod}. 
Moreover, we can apply Theorems \ref{th_6} and \ref{th_4_cons} with $v$ and
$v_{h,k}$ replaced by $v^r$ and $v^r_{h,k}$.

As in \cite{Alla_Falcone_Volkwein}, for any node $\color{red}\hat y^{i}_{r}\in \overline \Omega_r$ we set
$$
\color{red}\hat y^i=\bvar \hat y^{i}_{r},\quad i=1,\ldots n_s,
$$
and define
\begin{equation}\label{esta_si}
\tilde  v_{h,k}^r(y)=v_{h,k}^r(P^r_c y),\quad \forall y\in \overline \Omega.
\end{equation}
For $\color{red}\hat y^i$, $i=1,\ldots n_s$, by definition, we have
\begin{eqnarray*}
\tilde v_{h,k}^r(\hat y^i)=v_{h,k}^r(P^r_c \hat y^{i})=v_{h,k}^r(\hat y^{i}_{r}),
\end{eqnarray*}
since $\color{red}\hat y^{i}_{r}\in {\Bbb R}^r$ are the coordinates of $\color{red}\hat y^i\in {\Bbb R}^n$ respect to the basis functions of ${\cal \bV}^r$ \eqref{Ur}, see 
\eqref{defi1}, \eqref{defi2}.

Taking into account that $\color{red}\hat y^{i}_{r}+hf^r(\hat y^{i}_{r},u)$=$P_c^r(\hat y^i+hf(\hat y^i,u))$  and $g^r(\hat y_r^{i},u)=g(\hat y^i,u)$ then \eqref{fully_discrete_pod} can also be written as (see \cite[(4.17)]{Alla_Falcone_Volkwein})
\begin{equation*}\label{fully_discrete_pod_ini}
\tilde v_{h,k}^r(\hat y^{i})=\min_{u\in {U}_{\rm ad}}\left\{(1-\lambda h)\tilde v_{h,k}^r(\hat y^{i}+hf(\hat y^{i},u))+hg(\hat y^{i},u)\right\},\ i=1,\ldots,n_s.
\end{equation*}
Nevertheless, we do not use the above characterization of the fully discrete value function to bound the error.
\subsection{Error analysis of the method}
To prove the main results of the paper we need a previous lemma that we now state and prove. \textcolor{red}{Lemma \ref{nuevo}} bounds the difference between 
the value function solving the original problem \eqref{eq_v} and the value function solving the reduced order problem \eqref{eq_v_pod}.
\begin{lema}\label{nuevo} Let $v$ and $v^r$ be the solutions of \eqref{eq_v} and \eqref{eq_v_pod}, respectively. For $y\in\overline \Omega$,  let $P^r y\in {\Bbb R}^n$ be the projection of $y$ onto ${\cal \bV}^r$ and let $P^r_c y\in {\Bbb R}^r$ be as defined in \eqref{defi1}. Then, the following bound holds
\begin{eqnarray}\label{valor_cont}
|v(P^r y)-v^r(P^r_c y)|&\le&\ L_g\int_0^\infty te^{(L_f-\lambda )t}\max_{0\le s\le t}\|(I-P^r)f(y(s),u(s))\| dt\\
&&\quad + L_g\int_0^\infty te^{(L_f-\lambda )t}\max_{0\le s\le t}\|(I-P^r)f(y_r(s),u^r(s))\| dt,\nonumber
\end{eqnarray}
where $u,u^r:[0,\infty)\rightarrow {\Bbb R}^m$ are the controls giving the minimum in \eqref{eq_v} and \eqref{eq_v_pod} (with initial conditions
$P^r y$ and $P^r_c y$, respectively), $y(t)$ is the solution of \eqref{din_sis} with $y(0)=P^r y$ and control $u(t)$, i.e. 
$\color{red}y=y(P^r y,u)$ and
$y_r(t)$ is the solution of \eqref{din_sis} with $y(0)=P^r y$ and control $u^r(t)$, i.e., $\color{red}y_r=y(P^r y,u^r)$.
\end{lema}
\begin{proof} We argue as in \cite[Lemma 2]{Javier_yo}. 
Let $w:[0,\infty)\rightarrow {\Bbb R}^m$ be a given control that we now fix and let $y(t)$ be the solution of \eqref{din_sis} with $y(0)=P^r y$ and control $w(t)$. Then
\begin{equation}\label{uno}
y(t)=P^r y +\int_0^t f(y(s),w(s))\,ds.
\end{equation}
Let $ y^r(t)$ be the solution of \eqref{din_sis_pos} with control $w(t)$ and $y^r(0)=P^r_c y$.
Then, recalling definitions \eqref{defi1} and \eqref{defi2} from which we obtain $\bvar P^r_c y=P^r y$, we can write
\begin{equation}\label{dos}
\bvar y^r(t)=P^r y+\int_0^t\bvar P^r_c f(\bvar y^r(s),w(s))\,ds.
\end{equation}
Subtracting \eqref{dos} from \eqref{uno} we get
\begin{eqnarray*}
y(t)-\bvar y^r(t)&=&\int_0^t (f(y(s),w(s))-\bvar P^r_c f(\bvar y^r(s),w(s)))\,ds\\
&=&\int_0^t(f(y(s),w(s))-\bvar P^r_c f(y(s),w(s)))\,ds\\
&&\quad+\int_0^t (\bvar P^r_c f(y(s),w(s))-\bvar P^r_c f(\bvar y^r(s),w(s)))\,ds.
\end{eqnarray*}
Taking norms 
\begin{eqnarray}\label{tres}
\|y(t)-\bvar y^r(t)\|&\le&\int_0^t \|(I-P^r)f(y(s),w(s))\|\,ds
\\
&&\quad+\int_0^t \|\bvar P^r_c f(y(s),w(s))-\bvar P^r_c f(\bvar y^r(s),w(s))\|\,ds.\nonumber
\end{eqnarray}
Using again that $\bvar P^r_c=P^r$ and applying  that $P^r$ is a projection together with \eqref{lip_f} we get
\begin{eqnarray*}
&&\|\bvar P^r_c f(y(s),w(s))-\bvar P^r_c f(\bvar y^r(s),w(s))\|^2\\
&&\quad=\| P^r \left(f(y(s),w(s))- f(\bvar y^r(s),w(s))\right)\|^2\\
&&\quad\le \|f(y(s),w(s))-f(\bvar y^r(s),w(s)\|^2\le  L_f^2 \|y(s)-\bvar y^r(s)\|^2.
\end{eqnarray*}
Going back to \eqref{tres}
$$
\|y(t)-\bvar y^r(t)\|\le L_f\int_0^t\|y(s)-\bvar y^r(s)\|ds +t\max_{0\le s\le t}\|(I-P^r)f(y(s),w(s))\|.
$$
Applying Gronwall's lemma we get
\begin{eqnarray}\label{yytilde}
\|y(t)-\bvar y^r(t)\|\le {e^{L_ft}}\left(t\max_{0\le s\le t}\|(I-P^r)f(y(s),w(s))\|\right).
\end{eqnarray}
We now observe that from definitions \eqref{eq:funcional} and \eqref{eq:funcional_pod} we get
\begin{eqnarray*}
|\hat J(P^r y,w)-\hat J^r(P^r_c y,w)|\le\int_0^\infty |g(y(t),w(t))-g(\bvar y^r(t),w(t))|e^{-\lambda t}dt.
\end{eqnarray*}
Applying then the Lipschitz-continuity property of $g$, \eqref{lip_g}, together with \eqref{yytilde} we get
\begin{eqnarray}\label{cota_J}
&&|\hat J(P^r y,w)-\hat J^r(P^r_c y,w)|\le\nonumber\\
&&\quad L_g\int_0^\infty t{e^{(L_f-\lambda )t}}\max_{0\le s\le t}\|(I-P^r)f(y(s),w(s))\|\ dt.
\end{eqnarray}
To conclude we will argue similarly as in in \cite[Theorem 4]{Javier_yo}. 

For any $y\in \overline \Omega$, let $u^r:[0,\infty)\rightarrow {\Bbb R}^m$ be the control giving the minimum in \eqref{eq_v_pod} with initial condition $P^r_c y$. Since by definition of $v$, $v(P^r y)\le \hat J(P^r y,u^r)$ and $v^r(P^r_c y)=\hat J^r(P^r_c y,u^r)$, applying \eqref{cota_J} with $w=u^r$ we get
\begin{eqnarray}\label{valor1}
v(P^r y)-v^r(P^r_c y)&\le& \hat J(P^r y,u^r)-\hat J^r(P^r_c y,u^r)\nonumber\\
&\le& L_g\int_0^\infty te^{(L_f-\lambda )t}\max_{0\le s\le t}\|(I-P^r)f(y_r(s),u^r(s))\|\ dt,
\end{eqnarray}
where $y_r(t)$ is the solution of \eqref{din_sis} with $y(0)=P^r y$ and control $u^r(t)$.

On the other hand, let $u:[0,\infty)\rightarrow {\Bbb R}^m$ be the control giving the minimum in \eqref{eq_v}
with initial condition $P^r y$. Arguing as in \eqref{valor1} and applying \eqref{cota_J} again, with $w=u$,  we get
\begin{eqnarray}\label{valor2}
v^r(P^r_c y)-v(P^r y)&\le& \hat J^r(P^r_c y,u)-\hat J(P^r y,u)\nonumber\\
&\le& 
L_g\int_0^\infty te^{(L_f-\lambda )t}\max_{0\le s\le t}\|(I-P^r)f(y(s),u(s))\|\ dt,
\end{eqnarray}
where $y(t)$ is the solution of \eqref{din_sis} with $y(0)=P^r y$ and control $u$. From \eqref{valor1} and \eqref{valor2} we conclude
\eqref{valor_cont}.
\end{proof}
In next theorem we bound the difference between the value function of the original problem and our fully discrete approximation
based on POD.
Let $u^r$ be the control giving the minimum in \eqref{eq_v_pod}. For the proof of next theorem we need to assume that there exists a positive constant $L_{u^r}>0$ such that
\begin{eqnarray}\label{lip_u_pod}
\|u^r(t)-u^r(s)\|_2\le L_{u^r} |t-s|.
\end{eqnarray}
\begin{Theorem}\label{th_1_impo}
Let $v$ be the solution of \eqref{eq_v} and let $\tilde v_{h,k}^r$ be its fully discrete POD approximation 
defined in \eqref{fully_discrete_pod}-\eqref{esta_si}. Assume conditions  \eqref{lip_f}, \eqref{lip_f2}, \eqref{infty_f}, \eqref{lip_g}, \eqref{lip_g2}, \eqref{cota_g}, \eqref{semi_con_f},  \eqref{semi_con_g} and \eqref{lip_u_pod} hold. Assume $\lambda>\max(2L_g,L_f,\overline L_r)$ with $\overline L_r=CrL_f$.  Then, for
any $y\in \overline \Omega$ and $0\le h\le 1/(2\lambda)$ there exist positive constants $C_1$ and $C_2$  such that the following bound holds for $y\in\overline \Omega$ 
\begin{eqnarray}\label{impor_1}
|v(y)-\tilde v_{h,k}^r(y)|&\le& C_1(h+k_r)+{C_2h}+\frac{L_g}{\lambda-L_f}\|y-P^r y\|\nonumber\\
&&\ + L_g\int_0^\infty te^{(L_f-\lambda )t}\max_{0\le s\le t}\|(I-P^r)f(y(s),u(s))\| dt\\
&&\ + L_g\int_0^\infty te^{(L_f-\lambda )t}\max_{0\le s\le t}\|(I-P^r)f(y_r(s),u^r(s))\| dt,\nonumber
\end{eqnarray}
where $u,u^r:[0,\infty)\rightarrow {\Bbb R}^m$ are the controls giving the minimum in \eqref{eq_v} and \eqref{eq_v_pod} (with initial conditions
$P^r y$ and $P^r_c y$, respectively), $y(t)$ is the solution of \eqref{din_sis} with $y(0)=P^r y$ and control $u(t)$, i.e., $\color{red}y=y(P^r y,u)$ and
$y_r(t)$ is the solution of \eqref{din_sis} with $y(0)=P^r y$ and control $u^r(t)$, i.e., $\color{red}y_r=y(P^r y,u^r)$.
\end{Theorem}
\begin{proof}
We first observe that adding and subtracting terms we get
\begin{eqnarray*}
|v(y)-\tilde v_{h,k}^r(y)|&\le& |v(y)-v(P^ry)|+|v(P^r y)-v^r(P^r_c y)|+|v^r(P^r_c y)-\tilde v_{h,k}^r(y)|.
\end{eqnarray*}
To bound the first term we write
$$
|v(y)-v(P^ry)|\le |v(y)-v_h(y)|+|v_h(y)-v_h(P^r y)|+|v_h(P^r y)-v(P^r(y)|,
$$
and then apply \eqref{cota_semi} to the first and third terms and the Lipschitz-continuity of $v_h$ that holds for 
$\lambda>L_f$ (see \cite[p. 473]{falcone11}) to the second term.
Then
$$
|v(y)-v(P^ry)|\le C h+\frac{L_g}{\lambda-L_f}\|y-P^r y\|.
$$
To bound the second term we apply Lemma \ref{nuevo}. 
To conclude we need to bound the third term. To this end we observe that
$$
v^r(P^r_c y)-\tilde v_{h,k}^r(y)=v^r(P^r_c y)-v_{h,k}^r(P^r_c y)
$$
so that we can  apply Theorem \ref{th_6} to  $v^r$ and $v^r_{h,k}$ to reach \eqref{impor_1}.
\end{proof}
\begin{remark}\label{reth4}
Let us observe that the first two terms on the right-hand side of \eqref{impor_1} give the rate of convergence of the method in terms of the time step $h$ and
mesh diameter $k_r$. The other three terms come from the POD approximation and will be bounded at the end of this section. These terms
depend on the set of snapshots and the tail of the eigenvalues in the singular value decomposition. 

To apply Theorem \ref{th_6} in the proof of Theorem \ref{th_1_impo} we use the properties of $f^r$ and $g^r$ that, as commented 
in Remark \ref{re_fr_gr} are inherited from the assumed hypothesis made on $f$ and $g$, \eqref{lip_f}, \eqref{lip_f2}, \eqref{infty_f}, \eqref{lip_g}, \eqref{lip_g2}, \eqref{cota_g}. We also need to assume  condition \eqref{lip_u_pod} holds for the control of the
reduced order problem. Condition \eqref{lip_u_pod} can be weakened and one can still get convergence assuming instead \textcolor{red}{the following convexity assumption}
\begin{itemize}
\item (CAr) For every $y^r\in {\Bbb R}^r$,
\begin{eqnarray*}
\left\{f^r(y^r,u), g^r(y^r,u),\quad  u\in U_{\rm ad}\right\}
\end{eqnarray*}
is a convex subset of ${\Bbb R}^{r+1}$.
\end{itemize}
This result is stated in Theorem \ref{th_2_impo} below. In the proof of the theorem we apply Theorem \ref{th_4_cons}
instead of Theorem \ref{th_6}. 
\end{remark}
\begin{Theorem}\label{th_2_impo}
Let $v$ be the solution of \eqref{eq_v} and let $\tilde v_{h,k}^r$ be its fully discrete POD approximation 
defined in \eqref{fully_discrete_pod}-\eqref{esta_si}. Assume conditions  \eqref{lip_f}, \eqref{lip_f2}, \eqref{infty_f}, \eqref{lip_g}, \eqref{lip_g2}, \eqref{cota_g}, \eqref{semi_con_f},  \eqref{semi_con_g} and {\rm(CAr)} hold. Assume $\lambda>\max(2L_g,L_f,\overline L_r)$ with $\overline L_r=CrL_f$.  Then, for
any $y\in \overline \Omega$ and $0\le h\le 1/(2\lambda)$ there exist positive constants $C_1$ and $C_2$ such that the following bound holds for $y\in\overline \Omega$ 
\begin{eqnarray}\label{cota_buena_cons_pod}
|v(y)-v_{h,k}(y)|&\le& C_1(h+k_r)+C_2 \frac{1}{(1+\beta)^2\lambda^2}(\log(h))^2h^{\frac{1}{1+\beta}}\nonumber\\&&\ +\frac{L_g}{\lambda-L_f}\|y-P^r y\|\\&&\ + L_g\int_0^\infty te^{(L_f-\lambda )t}\max_{0\le s\le t}\|(I-P^r)f(y(s),u(s))\| dt\nonumber\\
&&\ + L_g\int_0^\infty te^{(L_f-\lambda )t}\max_{0\le s\le t}\|(I-P^r)f(y_r(s),u^r(s))\| dt,\nonumber
\end{eqnarray}
where $\beta=\frac{\sqrt{r}L_f}{\lambda}$ and $u,u^r:[0,\infty)\rightarrow {\Bbb R}^m$ are the controls giving the minimum in \eqref{eq_v} and \eqref{eq_v_pod} (with initial conditions
$P^r y$ and $P^r_c y$, respectively), $y(t)$ is the solution of \eqref{din_sis} with $y(0)=P^r y$ and control $u(t)$, i.e., $\color{red}y=y(P^r y,u)$ and
$y_r(t)$ is the solution of \eqref{din_sis} with $y(0)=P^r y$ and control $u^r(t)$, i.e. $\color{red}y_r=y(P^r y,u^r)$.
\end{Theorem}
\begin{proof}
The proof is the same as the proof of Theorem \ref{th_1_impo} applying Theorem \ref{th_4_cons}
instead of Theorem \ref{th_6}.
\end{proof}
\begin{remark} The same comments as in Remark \ref{reth4} apply with the difference that the rate of convergence in terms of
the time step $h$ is reduced due to the weaker regularity requirements. Since $\beta$ is smaller than $1$ we loose at most
half an order in the rate of convergence in time up to a logarithmic term.
\end{remark}
To conclude we will give an estimation of the last three terms in \eqref{impor_1} and \eqref{cota_buena_cons_pod}.
The first term is bounded in the following lemma, where, recall, $p$ is the number of trajectories in the set of snapshots (see Section~\ref{Se:3.1}).
\begin{lema}\label{lema5}
For $y\in \overline \Omega$ and $P^r$ the $X$-orthogonal projection onto ${\cal \bU}^r$
the following bound holds
\begin{align}
\label{resto2}
\left\|y - P^ry\right\| \le &
 \min_{\genfrac{}{}{0pt}{2}{1\le l\le p}{0\le n\le M}} \left\| y-y^{l}(t_n) \right\|+
\Biggl[\left(3+24\frac{T^2}{\tau^2}\right)p\sum_{k={r+1}}^{d}\lambda_k\nonumber\\
&{}+\frac{16T}{3}\textcolor{red}{h^2}\max_{1\le l\le p}\int_0^T\|y^{l}_{tt}(s)\|^2\ ds\Biggr]^{1/2}.
\end{align}
\end{lema}
\begin{proof}
Let $k$ and $j$ be such that
$$
\left\| y-y^{j}(t_k) \right\| = \min_{\genfrac{}{}{0pt}{2}{1\le l\le p}{0\le n\le M}} \left\| y-y^{l}(t_n) \right\|.
$$
Then, we can write
\begin{equation}
\label{desco2}
y - P^ry =(I-P^r) \left(y - y^{j}(t_k) \right)+ y^{j}(t_k) - P^r y^{j}(t_k).
\end{equation}
Now noticing that $\left\| I-P^r\right\|\le 1$ and recalling~Lemma~\ref{le:poin}, from~\eqref{desco2} it follows
\eqref{resto2}
\end{proof}
\begin{remark}
The first term on the right hand-side in \eqref{resto2} reflects the closeness of the data $y$ to the set of snapshots while the other term is the projection error onto the POD basis. Let us observe that using time derivatives in the set of snapshots allow us to get a bound for this projection error in 
the discrete maximum norm in time, see Lemma~\ref{le:poin}. 
\end{remark}
\bigskip

We bound the second term on the right-hand side of \eqref{impor_1} in the following lemma. 
\begin{lema} \label{lema6} For $y\in \overline \Omega$ and $P^r$ the orthogonal projection onto ${\cal \bU}^r$ let $u:[0,\infty)\rightarrow {\Bbb R}^m$ be the control giving the minimum in \eqref{eq_v} with initial condition
$P^r y$. Let  $y(t)$ be the solution of \eqref{din_sis} with $y(0)=P^r y$ and control $u(t)$. Then, for any fixed $s\in[0,\infty)$
the following bound holds
\begin{eqnarray}\label{apply_deri}
\|(I-P^r)f(y(s),u(s))\|&\le& \min_{\genfrac{}{}{0pt}{2}{1\le l\le p}{0\le n\le M}}\|f(y(s),u(s))-f(y^l(t_n),u^l(t_n))\|\nonumber
\\
&&\quad+\frac{(M+1)}{\tau^2}p\sum_{k=r+1}^d\lambda_k. 
\end{eqnarray}
\end{lema}
\begin{proof}
We argue as in the proof of Lemma \ref{lema5}.
Let $k$ and $j$ be such that
$$
\|f(y(s),u(s))-f(y^j(t_k),u^j(t_k))\|=\min_{\genfrac{}{}{0pt}{2}{1\le l\le p}{0\le n\le M}}\|f(y(s),u(s))-f(y^l(t_n),u^l(t_n))\|.
$$
Then,
\begin{eqnarray*}
&&(I-P^r)f(y(s),u(s))=(I-P^r)(f(y(s),u(s))-f(y^j(t_k),u^j(t_k)))\\
&&\quad+(I-P^r)f(y^j(t_k),u^j(t_k)),
\end{eqnarray*}
so that applying \eqref{eq:cota_pod_deriv} and $\|I-P^r\|\le 1$ we get \eqref{apply_deri}.
\end{proof}
\begin{remark}
An error bound for $\|(I-P^r)f(y_r(s),u^r(s))\|$ can be obtain arguing exactly in the same way. 
\end{remark}
\bigskip

\begin{remark}As in \eqref{resto2}, the first term on the right-hand side of \eqref{apply_deri} reflects the closeness
of~$f(y(s),u(s))$ to the set of snapshots and the second one is the projection error onto the POD basis.
Let us observe
that the use of temporal derivates in the set of snapshots is essential to get the bound \eqref{apply_deri}. 

In case one has a uniform distribution along the discrete times of the errors in \eqref{eq:cota_pod_deriv}, as it is often the case
(at least in our experience with numerical computations concerning POD methods, see \cite[Figure 1]{Bosco_Volker_yo}), one would expect for the second term on the
right-hand side of \eqref{apply_deri} a behaviour as
$
\frac{p}{\tau^2}\sum_{k=r+1}^d\lambda_k
$
instead of the rude bound $\frac{(M+1)}{\tau^2}p\sum_{k=r+1}^d\lambda_k$.
\end{remark}
\textcolor{red}{
\begin{remark}
As in \cite{Alla_Falcone_Volkwein}, \cite{maceneay2}, \cite{Alla_et_all}, \cite{Javier_yo} we do not provide in this paper error bounds for the reduced control
whose values at the nodes are obtained solving \eqref{fully_discrete_pod}. Although this would be interesting, we are not aware
of similar error bounds in the literature. Actually, we think that the starting point should be getting those bounds for the controls
in the fully discrete scheme of the original (not reduced) method. We remark that the results in \cite{Javier_yo}, error bounds
of the fully discrete problem, in which
the theory of the present paper is based, are very recent. Actually, the bounds in \cite{Javier_yo} represent an improvement
in the error bounds of the fully discrete method over previous results obtained more than 25 years ago. 
In \cite[Section 1.2]{falcone11} the reconstruction of approximate optimal controls is considered comparing the
fully discrete case with the semi-discrete in time case. Getting bounds for the
computed controls of the fully discrete problem respect to the original problem could be an interesting subject of future research, for which at the moment we do not know if the techniques in \cite[Section 1.2]{falcone11} could be extended. To complete the discussion we include below a heuristic argument concerning the convergence of the approximate controls in the fully discrete case. For any node in the triangulation, 
$\hat y^i$, let us denote by $u_{h,k}^i$ the control giving the minimum in \eqref{fully_discrete}. On the other hand, let us denote by
$u^i$ the optimal control in the HJB equation \eqref{HJB} for $y=\hat y^i$. 
As stated in \cite[Section 1.2]{falcone11}, the controls could not be unique but one can select the control with minimum norm.
Now, let us observe that for the discrete value function it holds
\begin{equation}\label{discre1}
v_{h,k}(\hat y^i+hf(\hat y^i,u^i_{h,k}))=v_{h,k}(\hat y^i)+hf(y^i,u^i_{h,k})\cdot \nabla v_{h,k}(\hat y^i)+O(h^2).
\end{equation}
Taking into account that 
\begin{equation}
\label{discre2}
v_{h,k}(\hat y^i)=(1-\lambda h)v_{h,k}(\hat y^i+hf(\hat y^i,u^i_{h,k}))+hg(\hat y^i,u^i_{h,k}),
\end{equation}
and inserting \eqref{discre1} into \eqref{discre2} we get
\begin{eqnarray*}
v_{h,k}(\hat y^i)&=&(1-\lambda h)\left(v_{h,k}(\hat y^i)+hf(y^i,u^i_{h,k})\cdot \nabla v_{h,k}(\hat y^i)+O(h^2)\right)\\
&&\quad+hg(\hat y^i,u^i_{h,k}).
\end{eqnarray*}
And then
\begin{eqnarray*}
\lambda h v_{h,k}(\hat y^i)=hf(\hat y^i,u^i_{h,k})\cdot \nabla v_{h,k}(\hat y^i)+hg(\hat y^i,u^i_{h,k})+O(h^2).
\end{eqnarray*}
From which
\begin{eqnarray*}
\lambda v_{h,k}(\hat y^i)=f(\hat y^i,u^i_{h,k})\cdot \nabla v_{h,k}(\hat y^i)+g(\hat y^i,u^i_{h,k})+O(h).
\end{eqnarray*}
Now, since
\begin{equation}\label{ennodo}
\lambda v(\hat y^i)=f(\hat y^i,u^i)\cdot \nabla v(\hat y^i)+g(\hat y^i,u^i),
\end{equation}
and $v_{h,k}(\hat y^i)\rightarrow v(\hat y^i),$ for $h,k\rightarrow 0$
we obtain
\begin{equation}\label{alfi_si}
f(\hat y^i,u^i_{h,k})\cdot \nabla v_{h,k}(\hat y^i)+g(\hat y^i,u^i_{h,k})\rightarrow f(\hat y^i,u^i)\cdot \nabla v(\hat y^i)+g(\hat y^i,u^i).
\end{equation}
Arguing as in \cite[Section 1.2]{falcone11},
let us define
$$
L(y,u)=\frac{1}{\lambda}\left(f(y,u)\cdot\nabla v(y)+g(y,u)\right),
$$
and let us associate with $y$ a (unique) control $u(y)$ such that
$$
L(y,u(y))=\min_{u\in U_{\rm ad}}L(y,u)=v(y).
$$
Assume
\begin{equation}\label{asumi}
\nabla v_{h,k}(\hat y^i)\rightarrow \nabla v(\hat y^i)
\end{equation}
(which we have not proved) and $u_{h,k}^i\rightarrow \overline u^i$ for $h,k\rightarrow 0$.
Then, on the one hand, from \eqref{alfi_si},
$$
L(\hat y^i,u_{h,k}^i)\rightarrow L(\hat y^i,u^i),
$$
and, on the other
$$
L(\hat y^i,u_{h,k}^i)\rightarrow L(\hat y^i,\overline u^i)
$$
which implies $\overline u^i=u^i$ and $u_{h,k}^i\rightarrow u^i$ for $h,k\rightarrow 0$.
Finally, let us observe that the argument in \cite[Section 1.2]{falcone11} had already proved that 
for any fixed $h$ and $k\rightarrow 0$ the fully discrete controls $u_{h,k}^i$ converge to the corresponding semi-discrete
time control defined in \eqref{discrete_HJB}, for that value of $h$.
\end{remark}
}
\section{Numerical Experiments}

We now present some numerical experiments. We closely follow those in~\cite{Alla_Falcone_Volkwein} so that the new method we propose
can be compared with the method in ~\cite{Alla_Falcone_Volkwein}. \textcolor{red} {The authors in \cite{Alla_Falcone_Volkwein}
apply state snapshots in the reduced order method instead of snapshots based on time derivatives. We observe that, as explained in detail in the introduction, in the last case it is not necessary to consider both, state snapshots and time derivatives, since it has already been proved that only with time derivatives optimal bounds can be obtained. We observe that we have chosen the closed-loop
control type approach instead of the open-loop control type approach in the present paper. Also, the theory of the present paper develops the first approach. We do not compare the present method with methods based on the first approach since
our aim is just to propose, analyze and check in practice a new method that could be better or not (probably depending
on the examples) than other methods
 in the
literature. The numerical experiments of this section show that our method works fine in practice and is able to provide accurate
approximations.}

We first notice
that due to numerical reasons we have to choose a finite time horizon, so we select a sufficiently large $t_e>0$, which, in the experiments that follow, it was fixed to~$t_e=3$. As in~\cite{Alla_Falcone_Volkwein}, we consider the following convection-reaction-diffusion equation
\begin{equation}
\label{crd}
\begin{array}{rcll}
z_t-\varepsilon z_{xx} + \gamma z_x + \mu(z^3-z) & =& ub\quad &\hbox{\rm in $I\times (0,t_e)$},\\
z(\cdot,0)&=&z_0\quad &\hbox{\rm in $I$},\\
z(\cdot,t) &=&0\quad&\hbox{\rm in~$\partial I\times (0,t_e)$},
\end{array}
\end{equation}
with $\varepsilon=1/10$, and
where $I=(0,a)$ is an open interval, $z: I\times[0,t_e]\rightarrow {\mathbb R}$ denotes the state,  and $\gamma $ and~$\mu$ are positive constants. The controls~$u$ belong to the closed, convex, bounded set~${\mathbb U}_{\rm ad} = L^2(0,t_e,[u_a,u_b])$, for real values $u_a<u_b$. The cost functional to minimize is given by
\begin{equation}
\label{cost_f}
\int_{0}^{t_e} e^{-\lambda t}\left(\left\| z(\cdot,t,u)\right\|_{L^2(I)}^2 + \frac{1}{100}\left|u(t)\right|^2\right)\,dt,
\end{equation}
where we set $\lambda=1$. Notice then that in~\eqref{cost_f} the aim is to drive the state to~zero. 

We use a finite-difference method on a uniform grid of size~$\Delta x=l/N$  with $N=100$ on the interval $I=(0,l)$ to discretize~\eqref{crd} in space to obtain a system of ordinary differential equations (ODEs). To obtain the snapshots, the ODE system is integrated in time using
Matlab's command ode15s, which uses the numeric differentiation formulae (NDF)~\cite{NDF}, with sufficiently small tolerances for the local errors (below $10^{-12}$). The snapshots were obtained on a uniform (time) grid of diameter $1/20$. The time derivatives were obtained by evaluating the right-hand-side of the system of ODEs.

As in~\cite{Alla_Falcone_Volkwein}, equation~\eqref{fully_discrete_pod} 
$$
v_{h,k}^r(y^{i}_r)=\min_{u\in {U}_{\rm ad}}\left\{(1-\lambda h)v_{h,k}^r(y^{i}_r+hf^r(y^{i}_r,u))+hg^r(y^{i}_r,u)\right\},\ i=1,\ldots,n_s,
$$
was solved by fixed-point iteration, the stopping criterium being that two consecutive iterates differ in the maximum norm in less \textcolor{red}{than a given tolerance ${\rm TOL}_v$, initially set to}~$\textcolor{red}{{\rm TOL}_v={}}5\times 10^{-4}$. For the first iterate we choose
a family of constant controls $u^l$, $l=1,\ldots,p$ and at any point of the mesh $y_r^i$ (initial condition) we compute the approximate solution of \eqref{din_sis_pos} corresponding to this initial condition and control $u^l$. Then, we compute the value of the functional cost \eqref{fun_cos_pod}. Finally, the value of the initial iterate at $y_r^i$ is the minimum between the values of the functional cost for $l=1,\ldots,p$.

Once~\eqref{fully_discrete_pod} is solved, the optimal control
\begin{equation}
\label{suboptimal}
u_{h,k}^r(y^{i}_r)=\argmin_{u\in {U}_{\rm ad}}\left\{(1-\lambda h)v_{h,k}^r(y^{i}_r+hf^r(y^{i}_r,u))+hg^r(y^{i}_r,u)\right\},
\end{equation}
is obtained at any  mesh point $y^i_r$, $i=1,\ldots,n_s$. Then, the suboptimal feedback operator~$\Phi^r(y)$ is computed by interpolation. This means that for
$y\in \overline \Omega$ we project onto the POD space to get $P^r y$ and then 
$$
P^r y=\sum_{i=1}^{n_s}\mu_i y^i_r,\quad \Phi^r(y)=\sum_{i=1}^{n_s}\mu_i u_{h,k}^r(y^{i}_r),
$$
where the coefficients $\mu_i$ satisfy $0\le \mu_i\le 1$, $\sum_{i=1}^{n_s} \mu_i=1$.

With this, the closed-loop system
\begin{equation}
\label{closed_loop}
y'(t)=f(y(t),\Phi^r(y(t))), \qquad y(0)=y_0,
\end{equation}
is integrated, again, using the NDF formulae as implemented in Matlab's command ode15s with the same tolerances as in the computation of the snapshots. \textcolor{red}{We will see below that very different approximations to the solution of~\eqref{closed_loop} can be obtained with different values of the tolerance~${\rm TOL}_v$ for the fixed-point iteration solving~\eqref{fully_discrete_pod} (see Fig.~\ref{fig_TOLv}), so that we solved this equation for decreasing values of~${\rm TOL}_v$, each one 5 times smaller than the previous one until the relative error between the solutions of~\eqref{closed_loop} corresponding to two consecutive values of~${\rm TOL}_v$ was below~10\% (it usually turned out to drop dramatically from above $10\%$  to less than~0.01\%) . Here and in the sequel, by the relative error of a quantity~$\hat y$ with respect to~$y$ we mean $\left| y-\hat y\right| /\max(\left|y\right|,10^{-3})$. \textcolor{red}{For the optimal HJB states, for every value of time~$t$ for which the solution of~\eqref{closed_loop} was computed, we computed the maximum or the relative errors of the components of~$y(t)$. For~Test~2 in~Section~\ref{Test2}, due to the discontinous initial datum, it proved impossible to drive the relative error of two optimal HJB states computed with two different tolerances~${\rm TOL}_v$ below~10\%, so that we checked that the value of the relative errors measured in the norm~\eqref{norma} below was smaller than 10\%}.}

\textcolor{red}{With respect to the computational cost of solving~\eqref{cost_f} by fixed-point iteration, it is obviously proportional to the number of iterations, which, in the experiments below, ranges from as few as 131 for~$r=4$ in Test 2 in~Section~\ref{Test2} below, to as many as 1322 for $r=5$ in Test 3 in Section~\ref{Test3} below. On each iteration, the bulk of the cost is finding the nonnegative scalars $\mu_j^i$, $j=1,\ldots,n_s$, such that $y^{i}_r+hf^r(y^{i}_r,u)= \mu_1^i y_r^1 +\cdots+  \mu_{n_s}^i y_r^{n_s}$, which ranges from 70\% of the cost of the iteration for~$r=5$ in~Test 3 in~Section~\ref{Test3} to 95\% for $r=4$ in~Test~2 in~Section~\ref{Test2}, followed by the cost of obtaining~$f^r(y^{i}_r,u)$, $i=1,\ldots,{n_s}$, which ranges from
$4\%$ for  $r=4$ in~Test~2 in~Section~\ref{Test2} to~28\% for $r=5$ in~Test 3 in~Section~\ref{Test3}. We note that cost of obtaining~$f^r(y^{i}_r,u)$ can be substantially diminished using appropriate tensors or by means of techniques like discrete empirical inerpolation, which, for simplicity, we did not use in our codes.
}
\subsection{Test 1: Semilinear equation}
\label{Test1}

As in~\cite{Alla_Falcone_Volkwein}, we consider~\eqref{crd} with $\gamma=0$ and~$\mu=1$, $a=1$ and $b(x)=z_0(x)=2x(1-x)$. It is easy to check that the uncontrolled solution converges, as $t\rightarrow \infty$ to a non-null steady state (see also~\cite[Fig.~6.1]{Alla_Falcone_Volkwein}), and that the null solution is unstable.

For the finite-difference approximation, we consider $y:[0,t_e]\rightarrow {\mathbb R}^{N-1}$ with components $y_j(t)\approx z(x_j,t)$, $x_j=j\Delta x$, $j=1,\ldots,N-1$, \textcolor{red}{$\Delta x=1/N$,} solution of
\begin{equation}
\label{fd_rd}
Cy_t=\frac{1}{10}Ay + C(F(y) + uB)
\end{equation}
where the components of~$F$ and $B$ are, respectively $F_j = y_j(1-y_j^2)$,  $B_j=2 x_j(1-x_j)$, $j=1,\ldots,N-1$, and $A$ and~$C$  are $(N-1)\times (N-1)$ tridiagonal matrices matrices given by
\begin{equation}
\label{matricesAyC}
A=\frac{1}{(\Delta x)^2}\left[\begin{array} {cccccc} -2 &\ 1\\ \ 1& -2& \ 1\\ &\ddots &\ddots&\ddots\\ && \ 1& -2 &\ 1 \\ &&& \ 1 &-2\end{array}\right],
\qquad 
C=\frac{1}{12}\left[\begin{array} {cccccc} 10 &1\\ 1& 10& 1\\ &\ddots &\ddots&\ddots\\ &&1& 10 &1 \\ &&& 1 &10\end{array}\right],
\end{equation}
so that the finite-difference discretization~\eqref{fd_rd} is fourth-order convergent. The norm we consider in~${\mathbb R}^{N-1}$ is given by
\begin{equation}
\label{norma}
\left\| y\right\|^2 = \Delta x \sum_{j=1}^{N-1} y_j^2.
\end{equation}
Let as observe that this norm is an approximation to the integral $\int_0^1y(x)^2 dx$ of a function with values $y_j$ at the spatial mesh nodes. 

To compute the snapshots, as in~\cite{Alla_Falcone_Volkwein}, for constant controls $u\in U_{\rm snap}=\{-1,0,1\}$, we obtained the solutions $y^{(n)} =y(t_n)$ of~\eqref{fd_rd} every~$1/20$ time units, that is, for $t_n=n/20$, $n=1,\ldots 20t_e$, and then the time derivatives $y_t^{(n)}$ were computed from identity~\eqref{fd_rd}. For the reduced spaces, we consider the cases of POD basis with only $r=2,3$ and $4$ elements, also
as in~\cite{Alla_Falcone_Volkwein}. \textcolor{red}{The POD approximation~$y^r$ was then the mean of the snapshots plus a linear combination of the POD basis}. The control set ${\Bbb U}_{\rm ad}$ is given by \textcolor{red}{41} controls equally distributed in $[-1,1]$.

As in \cite{Alla_Falcone_Volkwein}, to define the domain $\overline \Omega^r$, we compute the projections of all
the snapshots. With this procedure we obtain a set of points in $\Bbb R^r$. Then, we define an hypercube containing this set of points. The aim of this procedure, in view of Lemma \ref{lema3}, is that the 
set $\overline \Omega^r$ defined in this way satisfies the invariance condition
\begin{equation}\label{new_inv}
y^r+hf^r(y^r,u)\in \overline \Omega^r, \quad y^r\in \overline\Omega^r, \quad u\in {\Bbb U}_{\rm ad}.
\end{equation}
The set~$\overline \Omega^r$ for $r=4$ was given by
\textcolor{red}{$$\overline \Omega^r =[-0.87,0.41]\times[-0.01,0.02]\times[-0.01,0.01]\times [-0.01,0.01].$$}
For this set we checked that condition \eqref{new_inv} holds.

We notice that our set $\overline \Omega^r$ is considerable smaller than the corresponding set in~\cite{Alla_Falcone_Volkwein} (see \cite[6.1.\  Test\ 1]{Alla_Falcone_Volkwein}) where the authors use the standard euclidean norm in~${\mathbb R}^{n}$ rather than the norm~\eqref{norma} we use here.
Since the domain is smaller we also consider partitions of~$\Omega^r$ smaller than those in \cite{Alla_Falcone_Volkwein}. We take maximum diameter
$k_r=0.01$, and, as in~\cite{Alla_Falcone_Volkwein}, we choose $h=0.1k_r$.

In Fig.~\ref{rd6} we have represented on top the optimal solution (left) for $r=4$, the difference between optimal solution with 4 and 2 POD basis functions (top middle) and the difference between optimal solution with 4 and 3 POD basis functions (top right). On bottom we have represented the optimal controls for $r=2$, $3$ and $4$.
\begin{figure}[h]
\begin{center}
\includegraphics[height=3.9truecm]{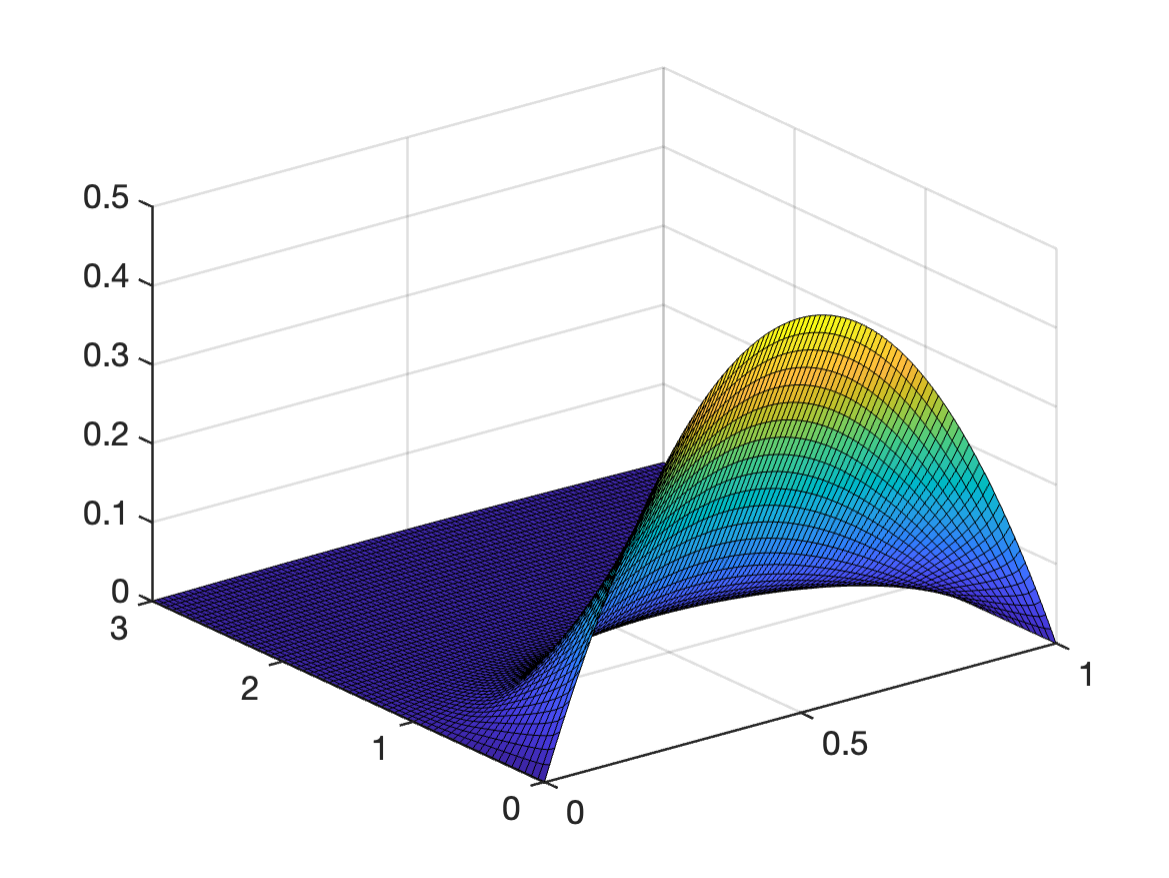}
\includegraphics[height=3.9truecm]{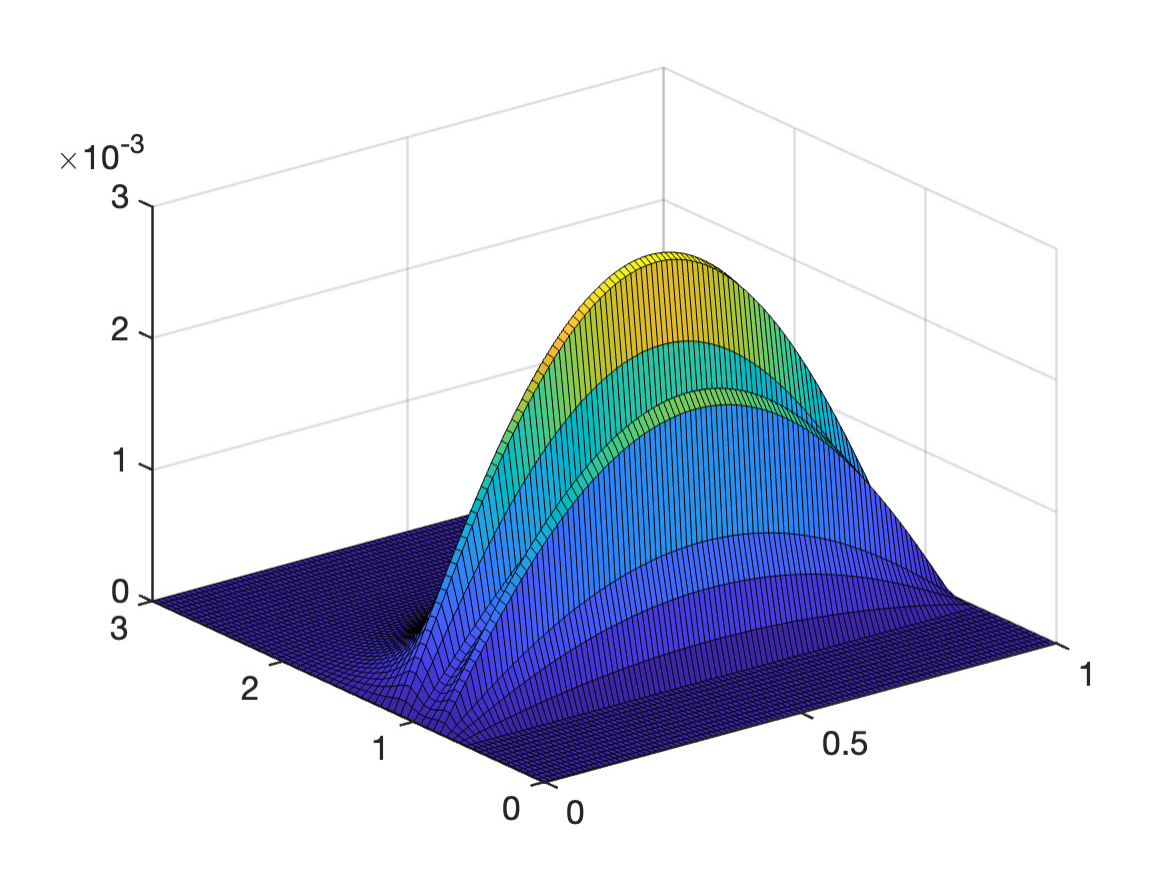}
\includegraphics[height=3.9truecm]{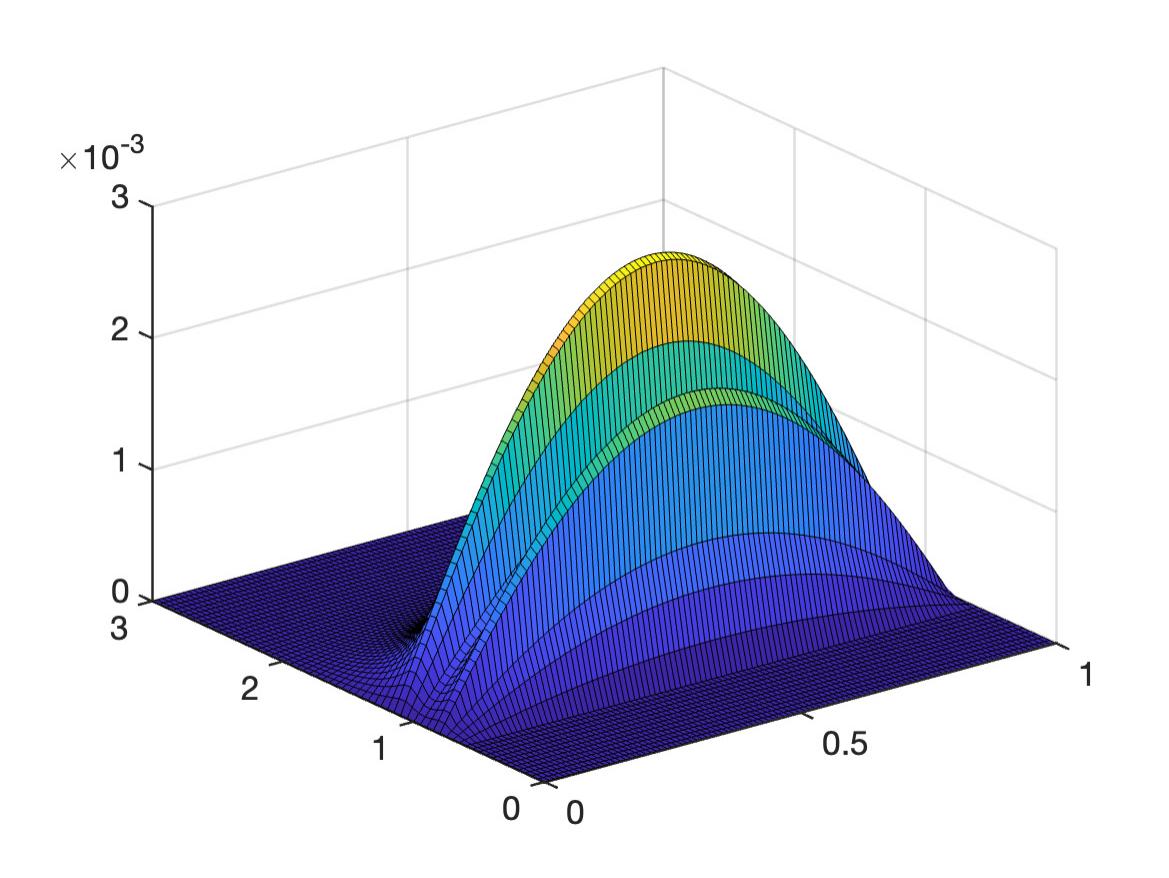}

\includegraphics[height=4truecm]{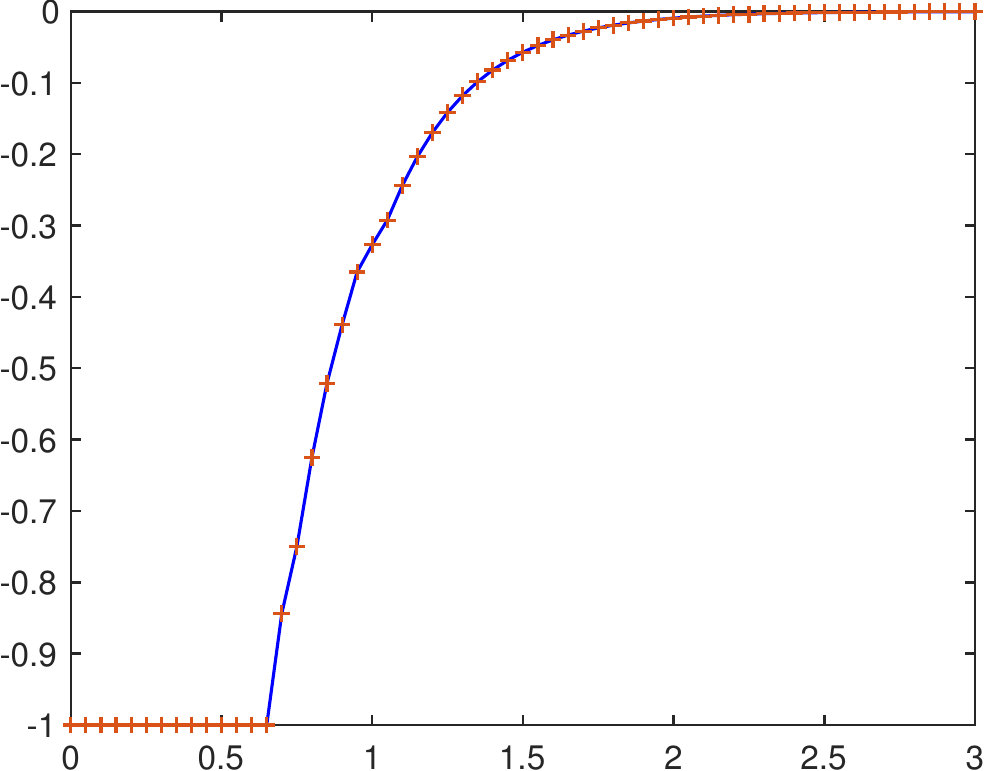}
\includegraphics[height=4truecm]{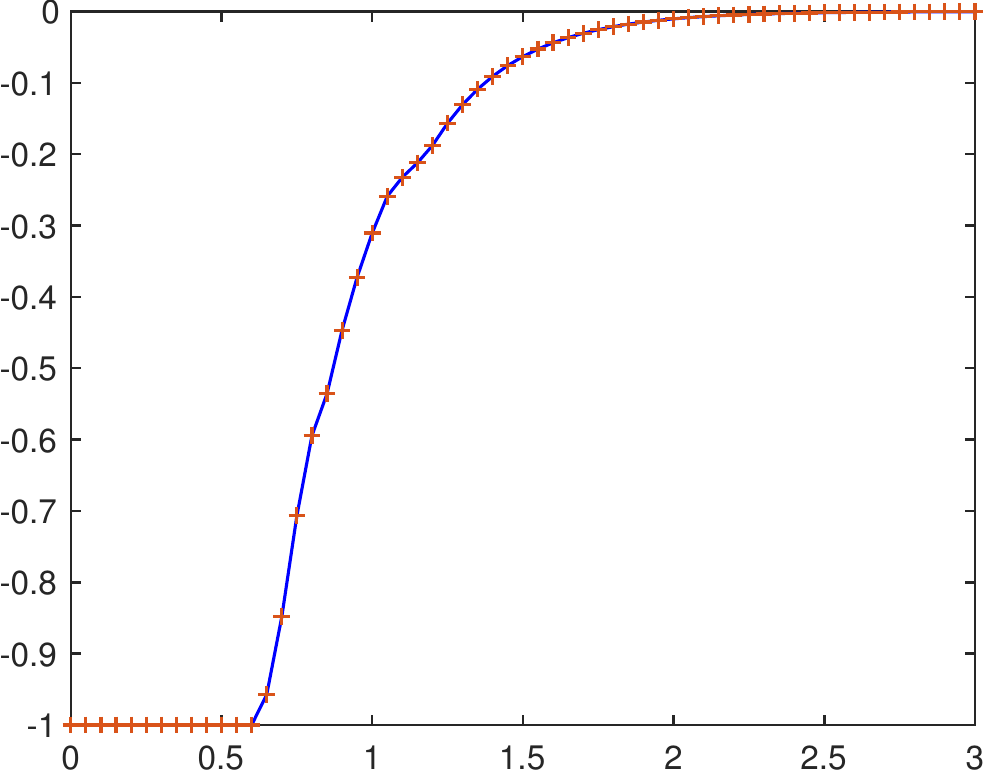}
\includegraphics[height=4truecm]{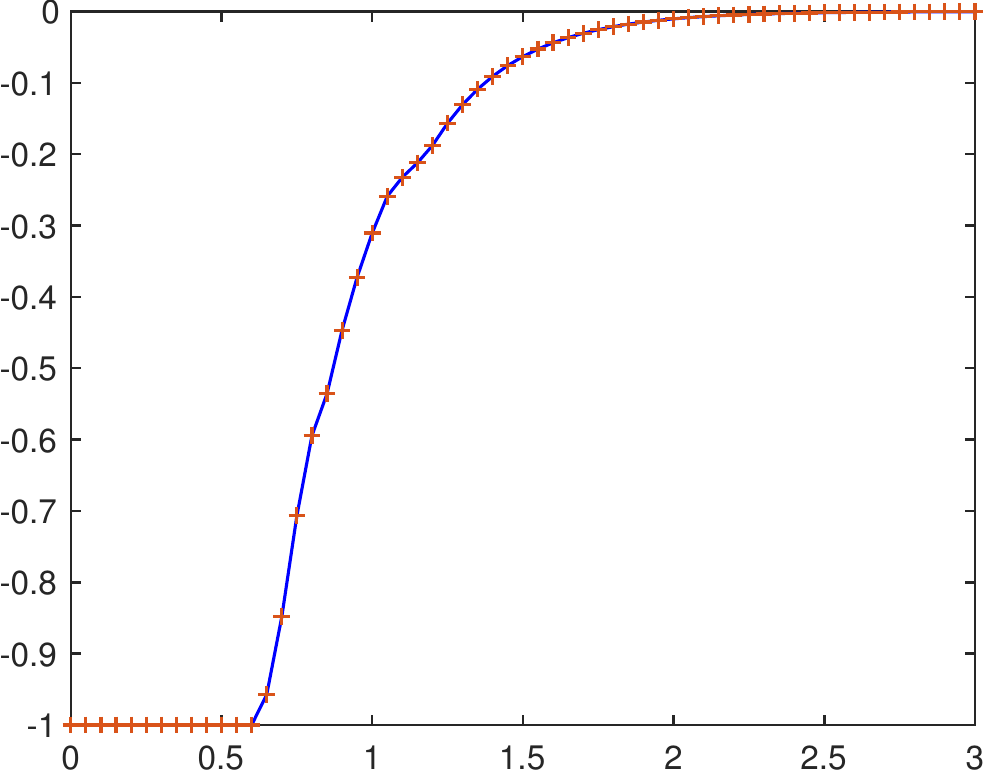}

\caption{\label{rd6} Test 1: Optimal HJB states computed with r=4 POD basis functions (top-left), difference between optimal solution with 4 and 2 POD basis functions (top-middle), difference between optimal solution with 4 and 3 POD bases (top-right). Optimal HJB controls with $r=4,3,2$ (bottom). \textcolor{red}{The red crosses correspond to the values of the controls that we have
joined by a blue line.}}
\end{center}
\end{figure}

We observe that we get much better results than those in \cite{Alla_Falcone_Volkwein}, although (apart from using a different set of snapshots)  in our method, both the finite-difference method and the time integrator that we use are more accurate than those in~\cite{Alla_Falcone_Volkwein}.
\textcolor{red}{We also notice that there is little discrepancy between the values of
the optimal HJB states for the different values of $r$ that we tried. We also computed the values of the cost functional~\eqref{cost_f} on the optimal HJB states for the three values of $r$. The values are shown in Fig.~\ref{rd_fig_cost_value}. It can be seen that the values decrease with~$r$ and that they differ in the ninth significant digit.
}
\begin{figure}[h]
\begin{center}
\includegraphics[height=2.5truecm]{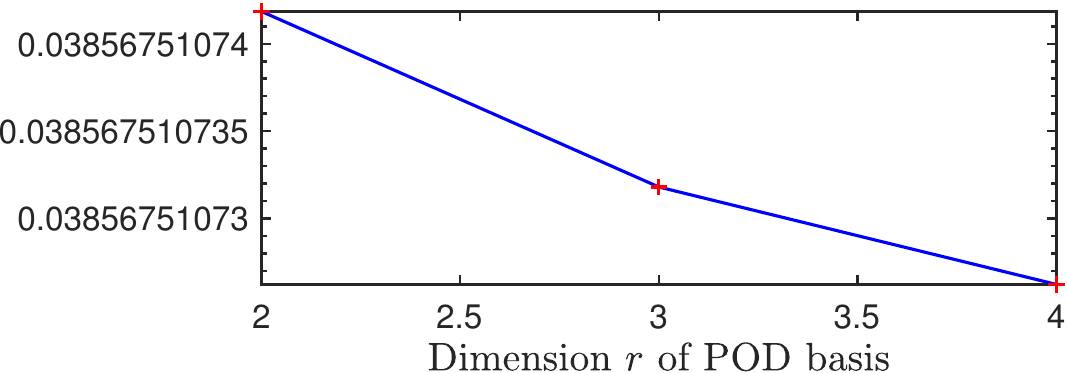}

\textcolor{red}{\caption{\label{rd_fig_cost_value} Test 1:  
Value of the cost functional~\eqref{cost_f} on the optimal HJB states for $r=2,3,4$. The red crosses correspond to the values of the cost values that are joined by a pdf line.
}}
\end{center}
\end{figure}

\textcolor{red}{As mentioned above there can be a significant difference between the optimal HJB states computed with solutions obtained by solving~\eqref{fully_discrete_pod} with different tolerances ${\rm TOL}_v$. In Fig.~\ref{fig_TOLv} we show the relative error between the optimal HJB states corresponding to tolerances~${\rm TOL}_v=5\times 10^{-4}$ and~${\rm TOL}_v=1\times 10^{-4}$ (left), and between this one and that corresponding to~${\rm TOL}_v=2\times 10^{-5}$ (centre). The right plot shows the corresponding optimal HJB controls. Fig.~\ref{fig_TOLv} shows the importance of solving~\eqref{fully_discrete_pod} accurately in order to obtain good optimal HJB states, thus, justifying that we computed the (approximations to the) solution of~\eqref{fully_discrete_pod} with decreasing values of~${\rm TOL}_v$ until the relative error of the corresponding optimal~HJB estates was below 10\%.
}
\begin{figure}[h]
\begin{center}
\includegraphics[height=3.9truecm]{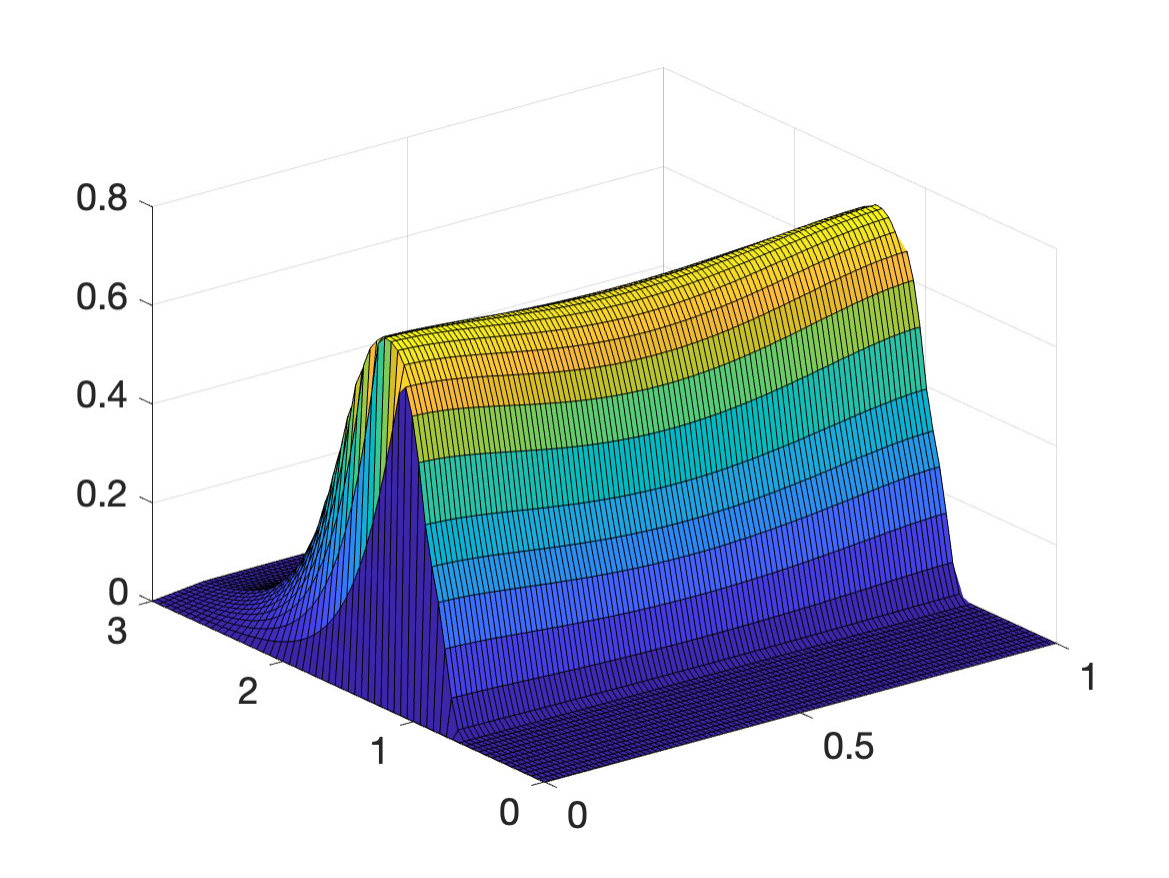}
\includegraphics[height=3.9truecm]{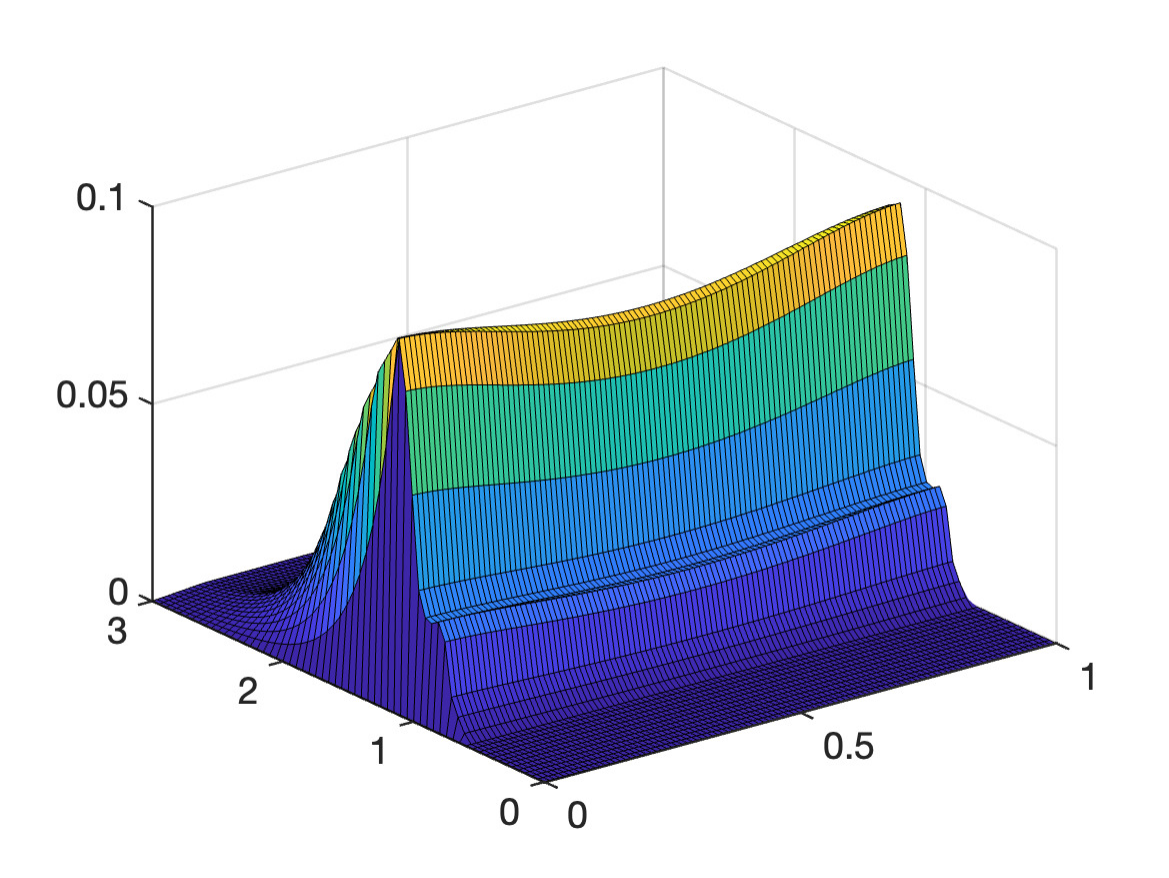}
\includegraphics[height=3.4truecm]{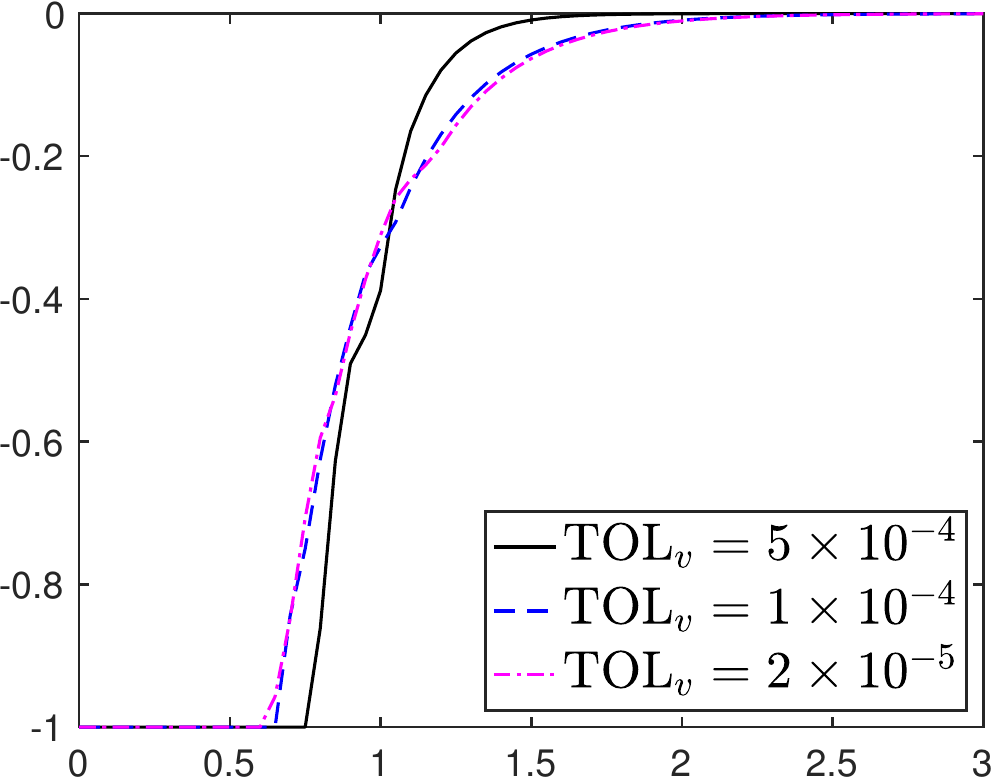}

\textcolor{red}{\caption{\label{fig_TOLv} Test 1:  Relative error between the optimal HJB states with $r=4$ corresponding to solving~\eqref{fully_discrete_pod} by fixed point iteration with tolerances~${\rm TOL}_v=5\times 10^{-4}$ and~${\rm TOL}_v=1\times 10^{-4}$ (left),  ${\rm TOL}_v=1\times 10^{-4}$ and~${\rm TOL}_v=2\times 10^{-5}$ (centre), and optimal HJB controls (right).}}
\end{center}
\end{figure}

\textcolor{red}{The results above suggest that, for this problem, it is enough with $r=3$. For this value of~$r$
we now check the effect of the finite-difference mesh in the optimal HJB states. In Table~\ref{table1} we show the relative errors of the optimal HJB states computed with $N=25$ and~$N=50$ with respect to that computed with $N=100$, as well as the relative errors of the corresponding controls. {For the controls, we show in~Table~\ref{table1} the maximum for all values of~$t\in\{0,0.05,0.1,\ldots,3\}$ of the relative errors, and for the states we show the maximum on the same values of~$t$ of the maximum
of the relative errors of the state on all the values of the corresponding spatial grid.} They confirm that the finite-difference discretization is of order~4. Due to the excellent accuracy obtained with $\Delta x=1/50$, the results that follow are done with that value of~$\Delta x$.
}
\begin{table}[h]
$$
\textcolor{red}{
\begin{array}{|c|c|c|c|c|}
\hline N & y&\hbox{\rm rate} &\Phi^r(y) &\hbox{\rm rate} \\ \hline
           25 & 7.24\times 10^{-5} & &  1.99\times 10^{-5}& \\ \hline
           50 & 4.25\times 10^{-6} & 4.09&  1.17\times 10^{-6} & 4.09\\ \hline
\end{array}}
$$
\textcolor{red}{\caption{\label{table1} Relative errors of the optimal HJB states and controls for $r=3$ computed with~$\Delta x=1/N$,  $N=25$ and~$N=50$, with respect to those computed with $N=100$.}}
\end{table}

\begin{figure}[h]
\begin{center}
\includegraphics[height=4truecm]{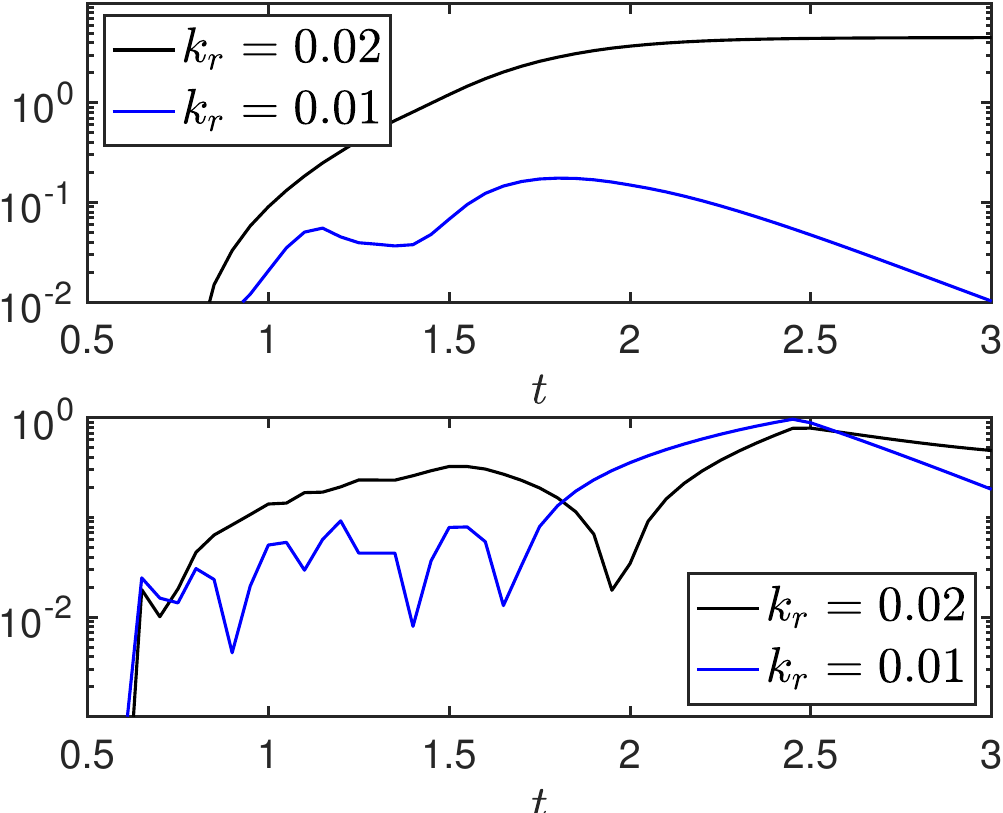}
\includegraphics[height=4truecm]{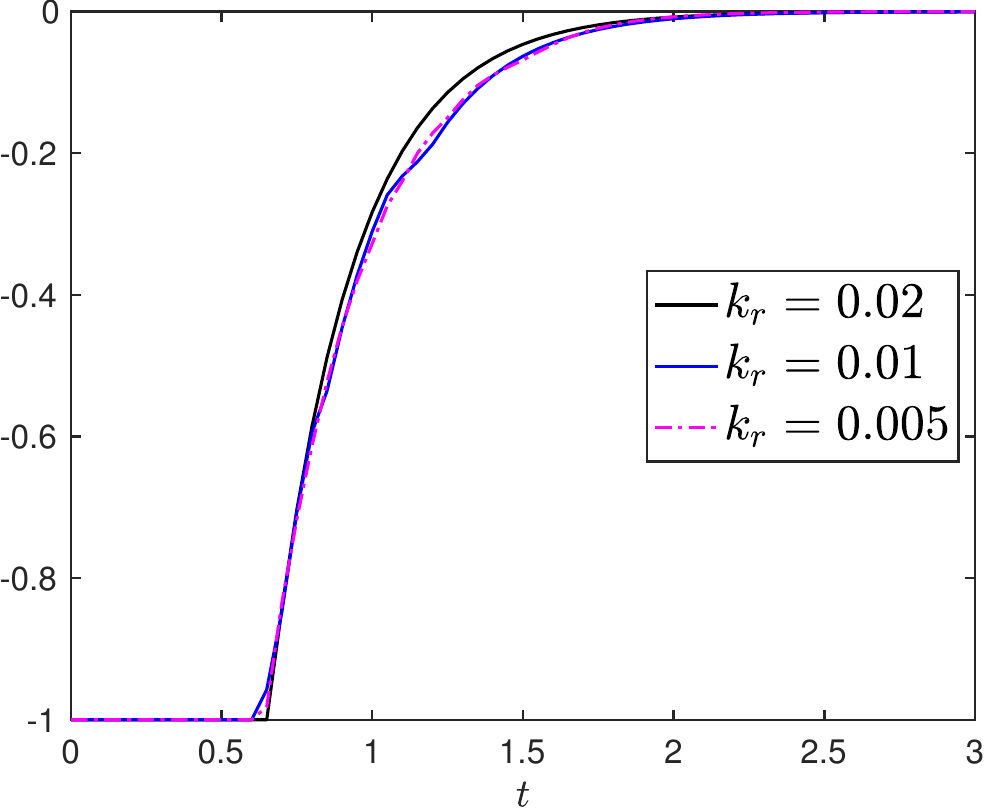}
\includegraphics[height=4truecm]{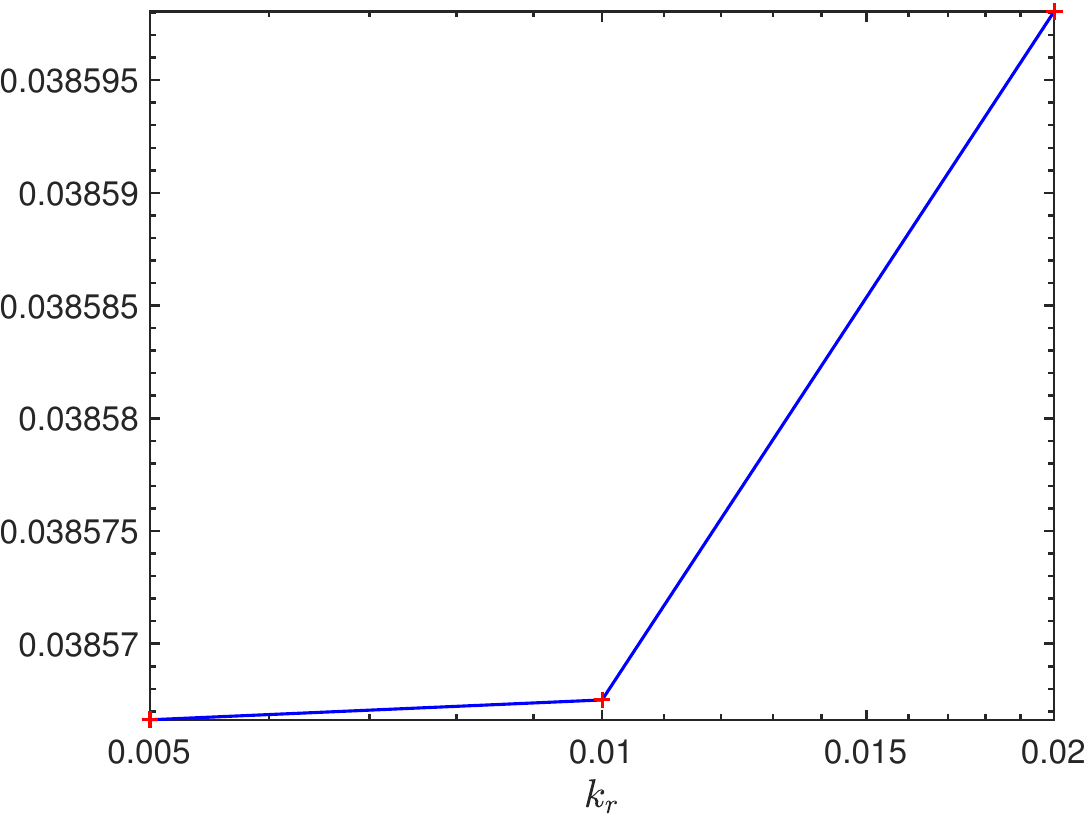}

\textcolor{red}{\caption{\label{fig_K} Test 1:  Results for different values of~$k_r$; relative errors between HJB states (top left) and controls (bottom left) with respect to to $k_r=0.005$;  HBJ controls (centre) and values of the cost functional~\eqref{cost_f} (right).}}
\end{center}
\end{figure}

\textcolor{red}{We now check the effect of different values $k_r$ of the diameter of the partition of~$\Omega^r$. To do this, we compare the results obtained with~$r=3$, $\Delta x=1/50$ and~$k_r=0.02, 0.01, 0.005$. In order not to spoil the better accuracy obtained with the smaller value of $k_r$ we took
${\Bbb U}_{\rm ad}$ with 161~controls equally distributed in $[-1,1]$ for~$k_r=0.005$, and, to simplify computations~with only 11 controls for~$k_r=0.02$ (we also try with 21 and 41 controls, but, although we do not have at present an explanation for it, using only 11 controls with $k_r=0.02$ gave somewhat better results). The results can be seen in~Fig.~\ref{fig_K} . The plots on the left show the {(maximum of the 51 points of the spatial grid of the)} relative errors of the optimal HJB states (top) and their controls (bottom) for~$k_r=0.02$ and~$k_r=0.01$ with respect to those of $k_r=0.005$, while the plot in the centre shows the HJB controls. The errors, as expected, are smaller for $k_r=0.01$ than for $k_r=0.02$, except for the controls for~$t\in[1.8,2.55]$ where they are slightly larger. We also notice that, for $k_r=0.01$, the relative errors remain below 10\% except for $t\in [1.5.2.2]$ (where they remain below 18\%) in the case of the errors in the optimal HJB states and $t\in[1.8,3]$ for the controls. Notice, however, that the largest errors take place where both the states and the controls are close to zero (recall the plots in Fig.~\ref{rd6}), were it is difficult to obtain small relative errors. Maybe this is the reason for the similar values of the cost functional~\eqref{cost_f} for the three values of~$k_r$, which are shown on the right plot in~Fig.~\ref{fig_K}; the relative errors
(with respect to~$k_r=0.005$)
 for~$k_r=0.02$ and~$k_r=0.01$ are 0.081\% and~0.0023\%, respectively.}

\textcolor{red}{One may wonder what is the result if the snapshots are used to obtain the POD basis as in~\cite{Alla_Falcone_Volkwein} instead of the time derivatives as in the present paper. Thus, we repeated our computations but replacing the time derivatives by the snapshots minus their mean. We did not find any significant difference. In Fig.~\ref{rd_u_ut} we show the results corresponding to~$\Delta x =1/50$, $r=3$, the set $\Omega^r$ being
$$
\Omega^r = (-0.42,0.9)\times (-0.01, 0.02)\times (-0.01,0.01).
$$
The optimal HJB control is shown on the right-plot, while the other two plots show the relative errors of the HJB state (left) and its control (centre) with respect to the results when the POD basis is taken from the time derivatives, $r=3$ and~$\Delta x=1/100$. We see that the relative errors are below 0.1\%, and thus, no difference can be seen between the right-plot in~Fig.~\ref{rd_u_ut} and the centre plot in~Fig.~\ref{rd6}. We also notice that our results when the POD basis is taken from the snapshots are better than those in~\cite{Alla_Falcone_Volkwein}. We believe that this is due to the higher accuracy of our computations (fourth-order convergent finite-difference method instead of a second-order convergent one, NDF with small tolerances to compute the snapshots instead of implicit Euler method, denser sets ${\Bbb U}_{\rm ad}$ for the control variable, smaller tolerance ${\rm TOL}_v$ in the fixed point method to solve~\eqref{fully_discrete_pod}, etc).
}

\begin{figure}[h]
\begin{center}
\includegraphics[height=4truecm]{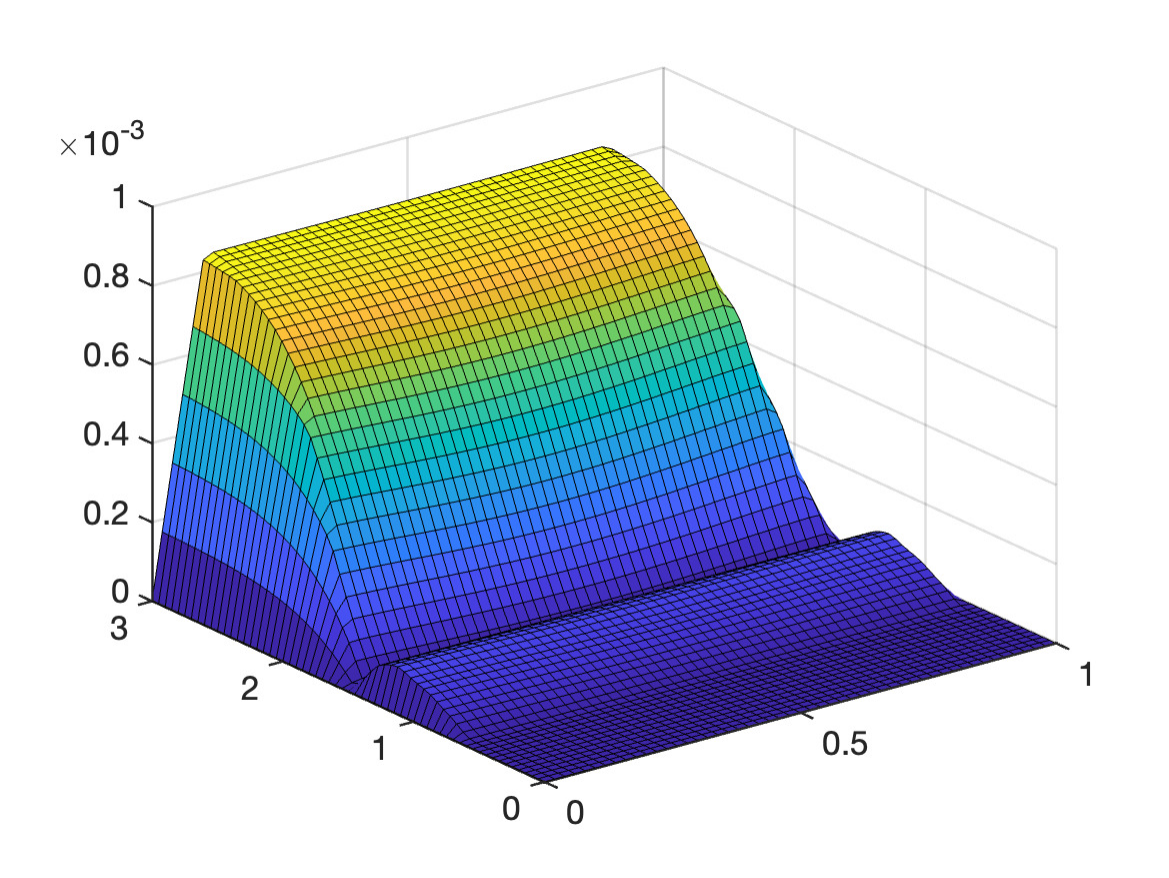}
\includegraphics[height=3.4truecm]{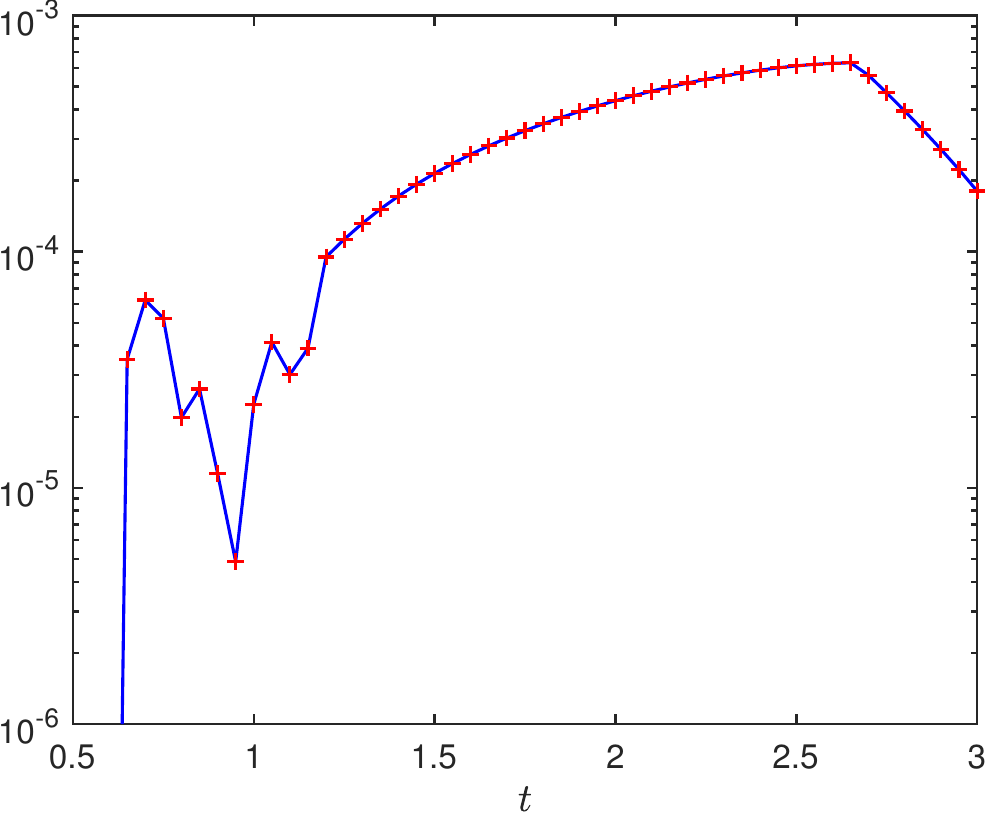}
\includegraphics[height=3.44truecm]{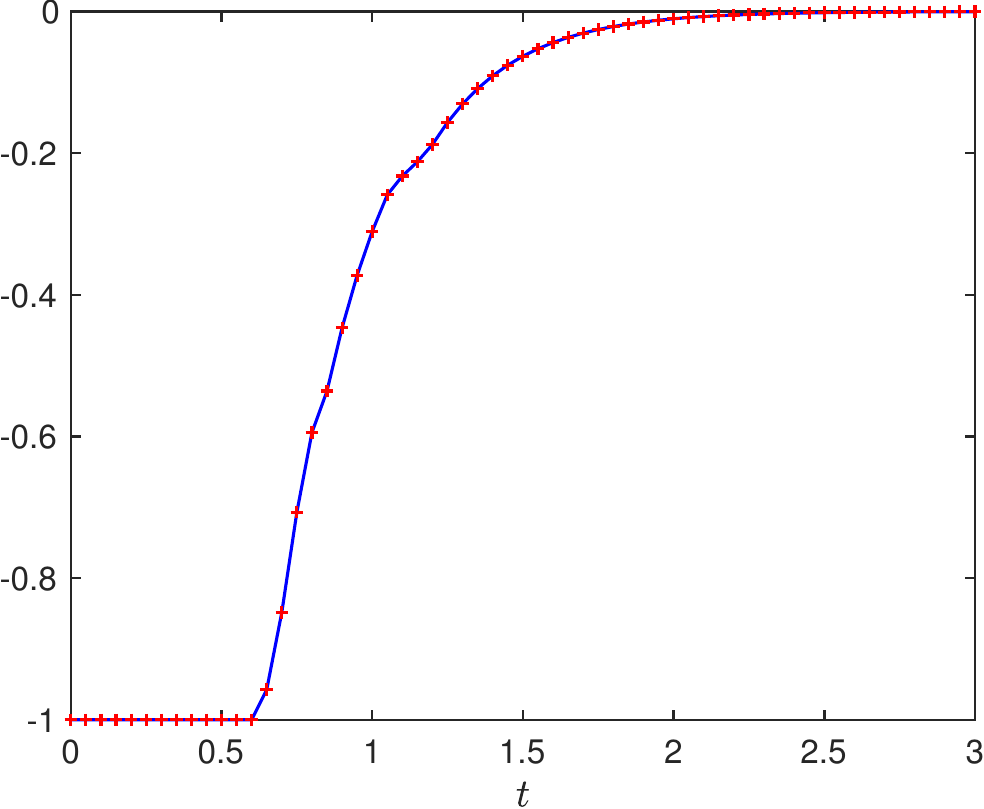}

\textcolor{red}{\caption{\label{rd_u_ut} Test 1:  Results for POD basis extracted from snapshots ($\Delta x =1/50$, $r=3$); Relative errors of HJB state (left) and control (centre) with respect to the case where POD basis is taken from time derivatives;  HBJ control (right).}}
\end{center}
\end{figure}

\textcolor{red}{The fact that very similar results are obtained when the POD basis is taken from the snapshots or the time derivatives should not be surprising. As shown in~\cite{Iliescu-Wang}, wether better results are obtained if the  POD basis is extracted from the snapshots or from their difference quotients is case-dependent and, as shown in~\cite{Bosco_Volker_yo2},
very similar results are usually obtained when the POD basis is taken from the time derivatives or the snapshots difference quotients. The advantage of using time derivatives or difference quotients for the POD basis is more from the theoretical side, since it allows to prove optimal convergence of the POD methods with less assumptions than when the POD basis is extracted from the snapshots. More recently, in~\cite{Bosco_Julia_pointwise}, it has been proved that using only snapshots for the POD basis, it is possible to prove error estimates for the corresponding POD methods with convergence rates as close to optimal as the smoothness of the solution from where the snapshots are taken allows. In view of the recent results in~\cite{Bosco_Julia_pointwise}, the analysis in the present paper
 can be easily adapted to cover also the case where POD basis is taken from the snapshots.
}

\subsection{Test 2: Advection-diffusion equation}
\label{Test2}

As in~\cite{Alla_Falcone_Volkwein}, we now consider~\eqref{crd} with $\gamma=1$ and $\mu=0$, $I=(0,2)$ and $z_0(x)=\max(0,0.5\sin(\pi x))$. We take $b$ as the characteristic function of the inverval~$(1/2,1)$. To compute the POD basis we compute the time derivatives of the states for constant controls $u=-2.2,-1.1,0$. The semidiscretization was done with a standard finite difference method
$$
\frac{dy_j}{dt}  =\frac{y_{j+1}-2y_j + y_{j-1}}{10(\Delta x)^2} - \frac{y_{j+1}-y_{j-1}}{2\Delta x} + b(x_j), \quad j=1,\ldots,N-1, \qquad y_0=y_N=0,
$$
which is second order convergent in problems with sufficiently smooth solutions.

Since the initial state~$z_0$ does not possess second-order derivatives in~$L^2$, we notice then that the time derivative $z_t$ blows up when $t\rightarrow 0$. For this reason, after the spatial discretization by finite differences, we replaced the time derivative at $t=0$ by the difference quotient~$(y^{(1)}-y^{(0)})/\Delta t$ of states at $t=0$ and $t=\Delta t$.  \textcolor{red}{Perhaps also for lack bounded time derivatives at $t=0$ and the more dissipative nature of the implicit Euler method, we found that, in the computation of the optimal HBJ states and controls, better results were obtained if the implicit Euler method with $\Delta t=1/20$ was used instead of the NDF with small tolerances.
Also, for reasons that we do not understand at present, we found that better results were obtained when the POD approximation was a linear combination of the POD basis plus  the initial condition $y_0$, instead of a linear combination of the POD basis plus the mean as in the previous section.} 

\textcolor{red}{In the previous test we had an invariance set so that we did not need to impose any boundary condition for solving
\eqref{fully_discrete_pod}.}
In this test we found it impossible to find a set~$\Omega^r$ satisfying condition~\eqref{new_inv} both when the POD basis is taken from the states and from their time derivatives. In this last case the set $\Omega^r$ for $r=4$ we used in the experiments was
$$\Omega^r=(-0.5,0.7)\times(-0.3,1.5)\times (-0.3,0.2)\times (-0.05,0.15).$$
 To overcome the lack of invariance of this set,  whenever for a vertex~$y^i_r$ we had $y^i_r+hf^r(y^i_r,u)\not\in \overline \Omega^r$, we simply replaced $y^i_r+hf^r(y^i_r,u)$ by its closest point on~$\partial\Omega^r$. This resulted in changing the value of~$y^i_r+hf^r(y^i_r,u)$ in less than 2\% in the first two coordinates in the POD basis and $15\%$ in the remaining ones, except for some negative values of the fourth coordinate where errors up to~$60\%$ were encountered. For example, an error of $15\%$ in the
 third coordinate  means that for some values of $y^i_r+hf^r(y^i_r,u)$ the third coordinate could be in the
 set $[-0.345,0.23]$ instead of $[-0.3,0.2]$. Nevertheless, as we will see below, the results obtained with the POD approximation in this test were excellent.

Since this problem is linear-quadratic, the solution of HJB equation can be computed by solving Riccati equation. In Fig.~\ref{cvcontrol} we show the uncontrolled solution (left), the optimal LQR state (middle) and the optimal LQR control (right).
\begin{figure}
\begin{center}
\includegraphics[height=3.8truecm]{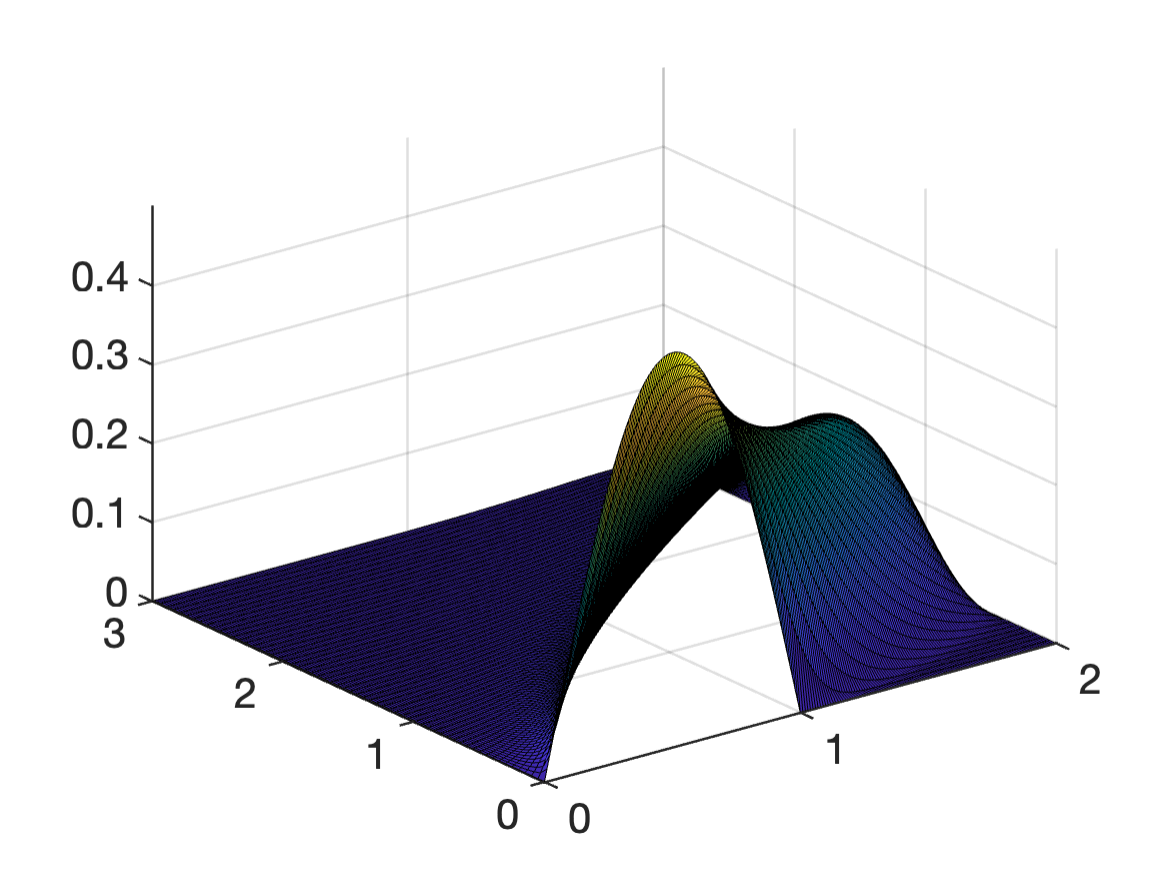}
\includegraphics[height=3.8truecm]{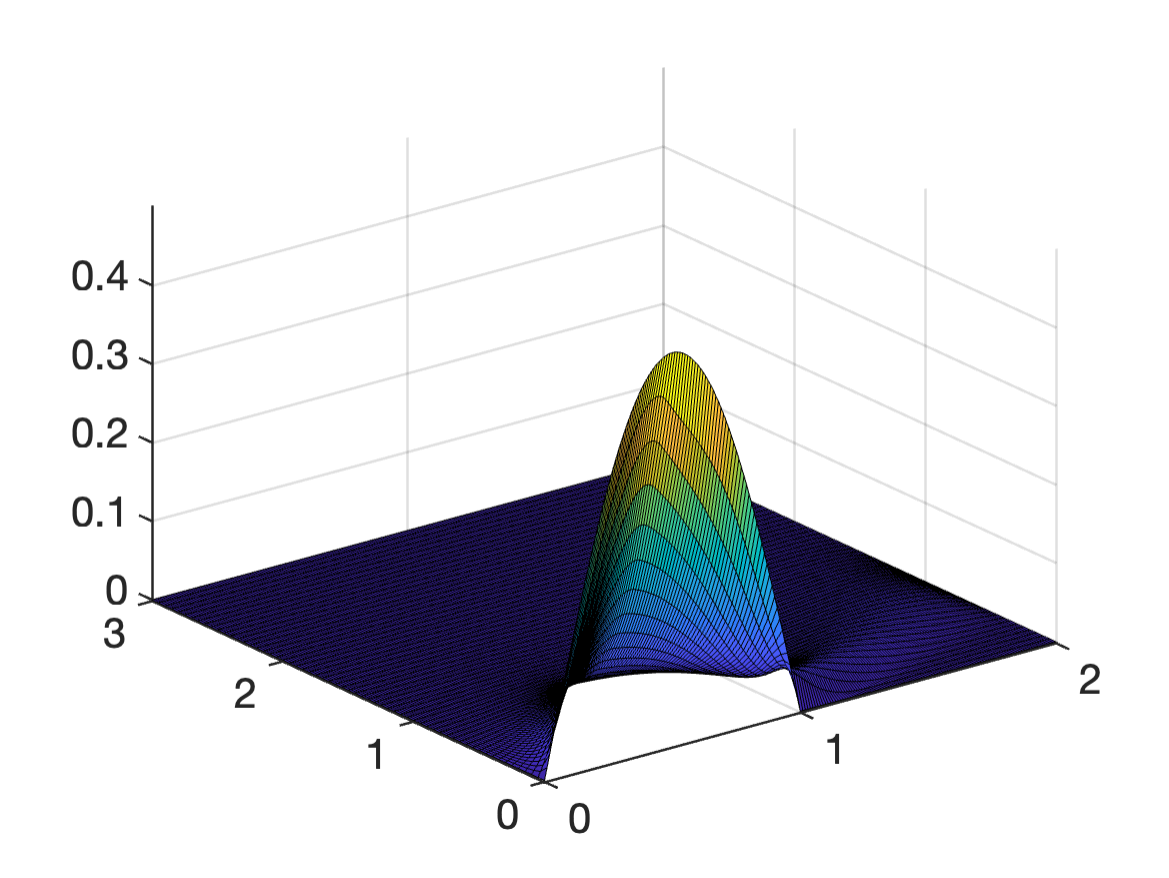}
\includegraphics[height=2.9truecm]{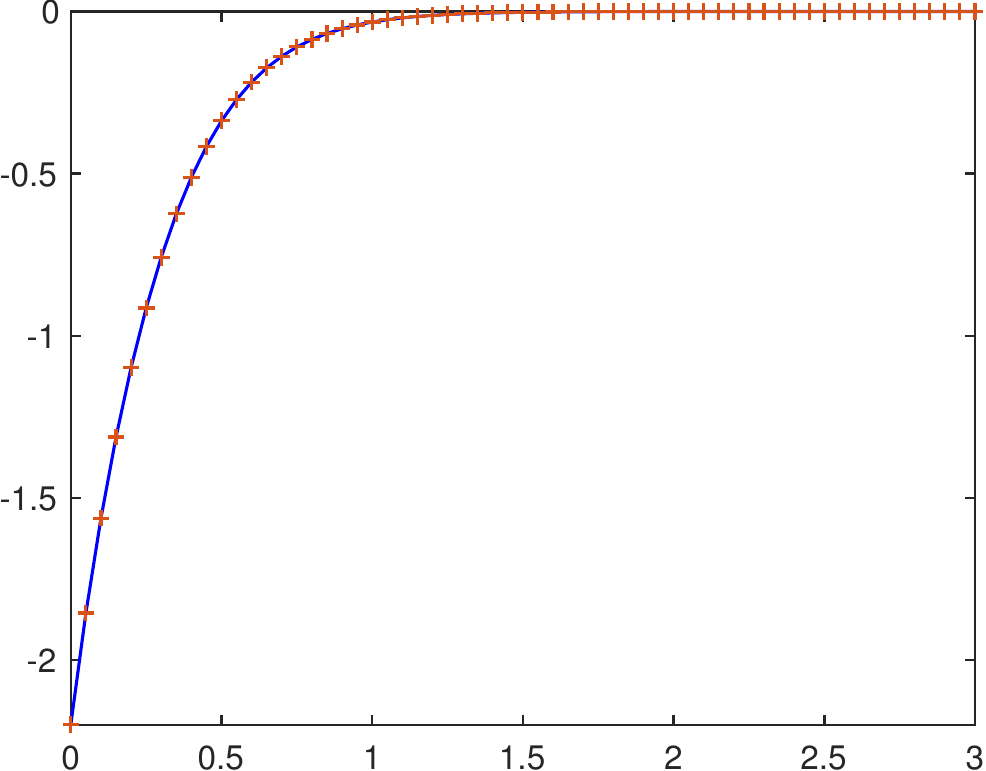}
\caption{\label{cvcontrol} Test 2: Uncontrolled solution (left), the optimal LQR state (middle) and the optimal LQR control (right).}
\end{center}
\end{figure}

\begin{figure}
\begin{center}
\includegraphics[height=3.9truecm]{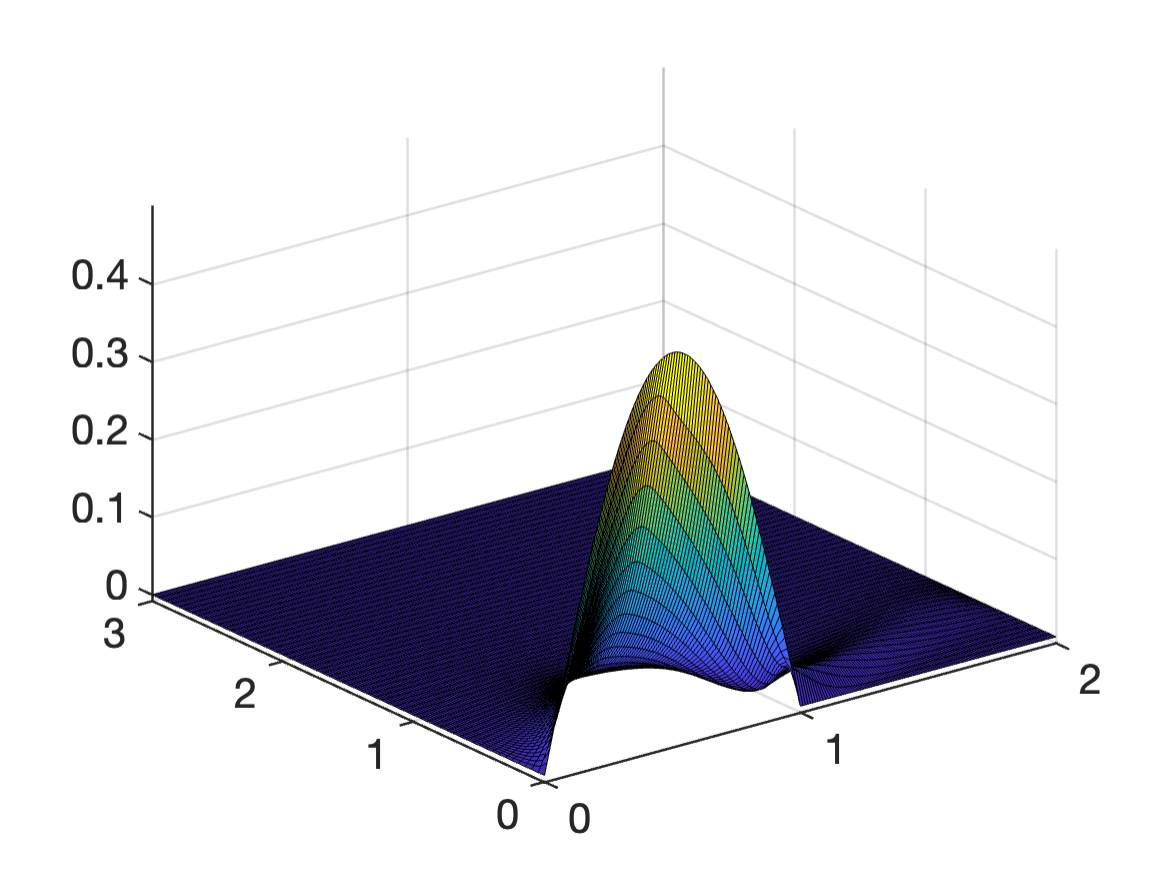}
\includegraphics[height=3.9truecm]{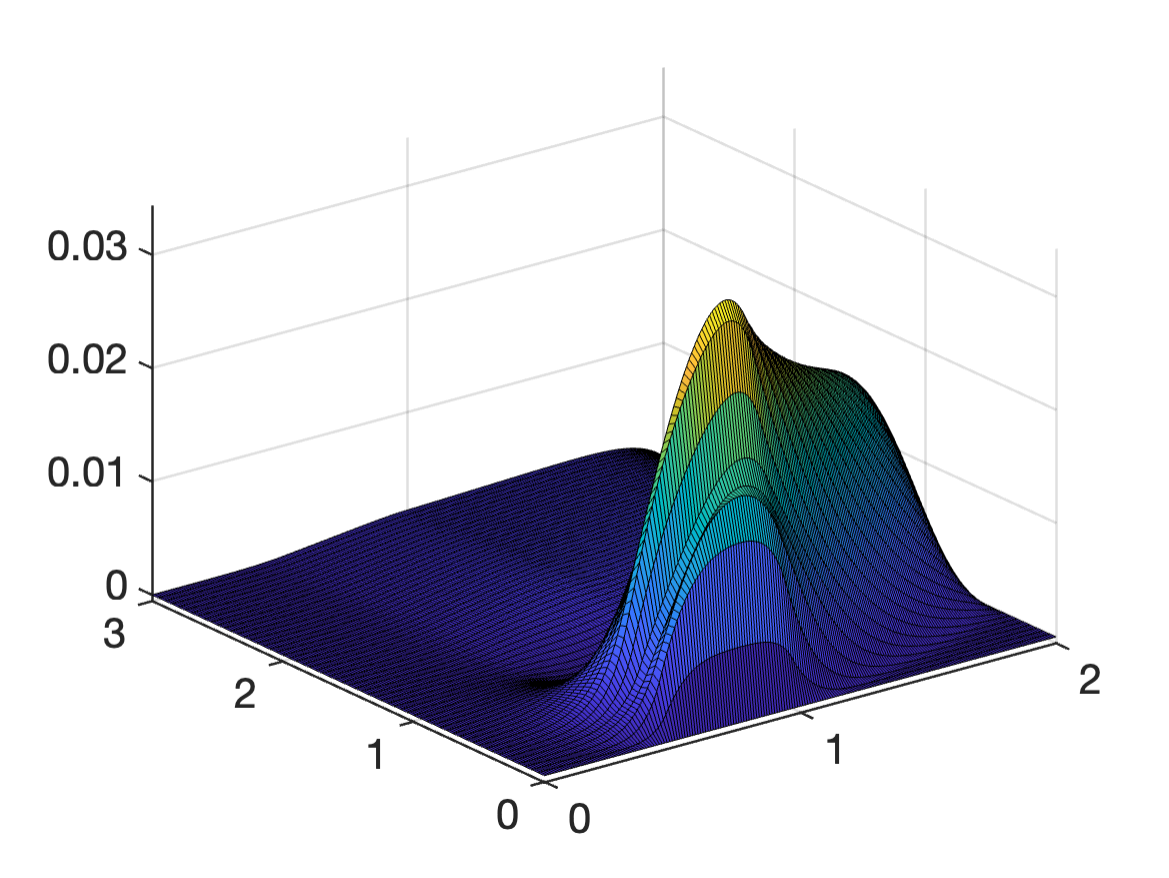}
\includegraphics[height=3.9truecm]{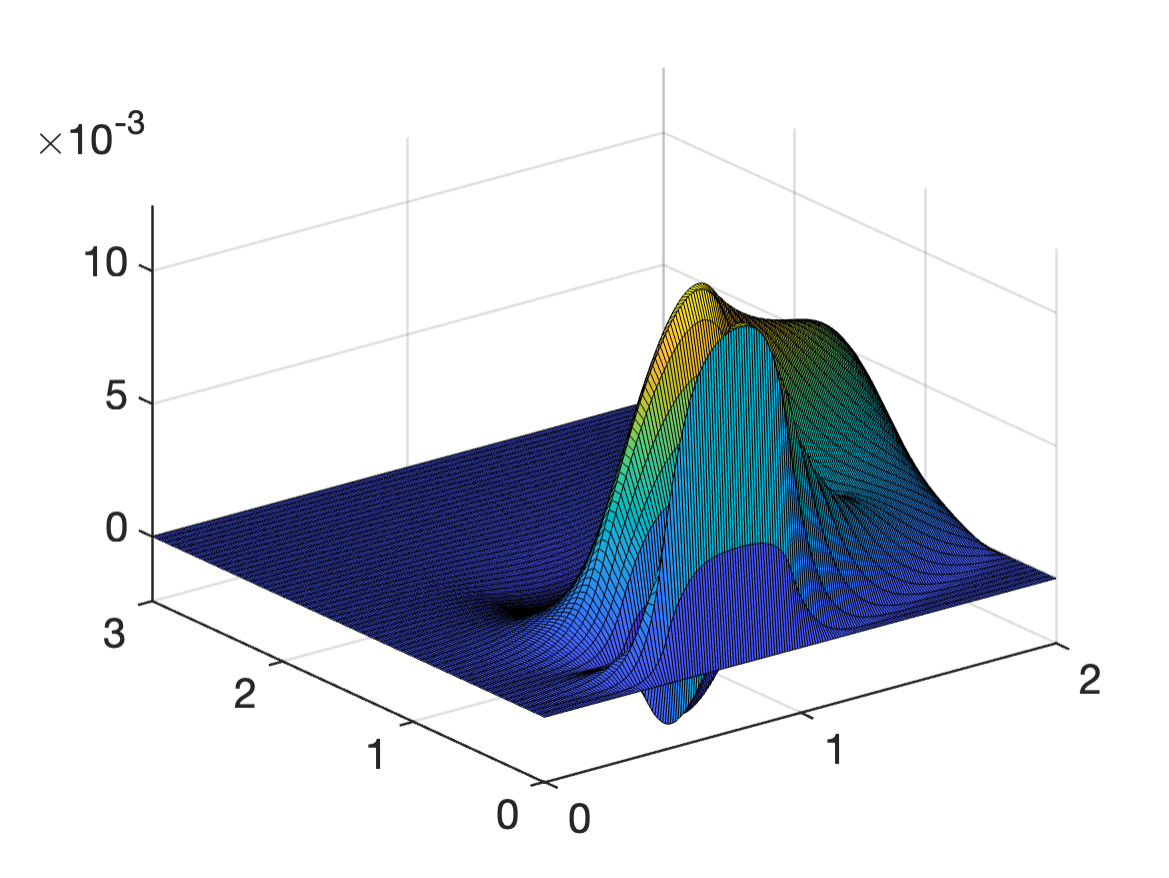}

\includegraphics[height=4truecm]{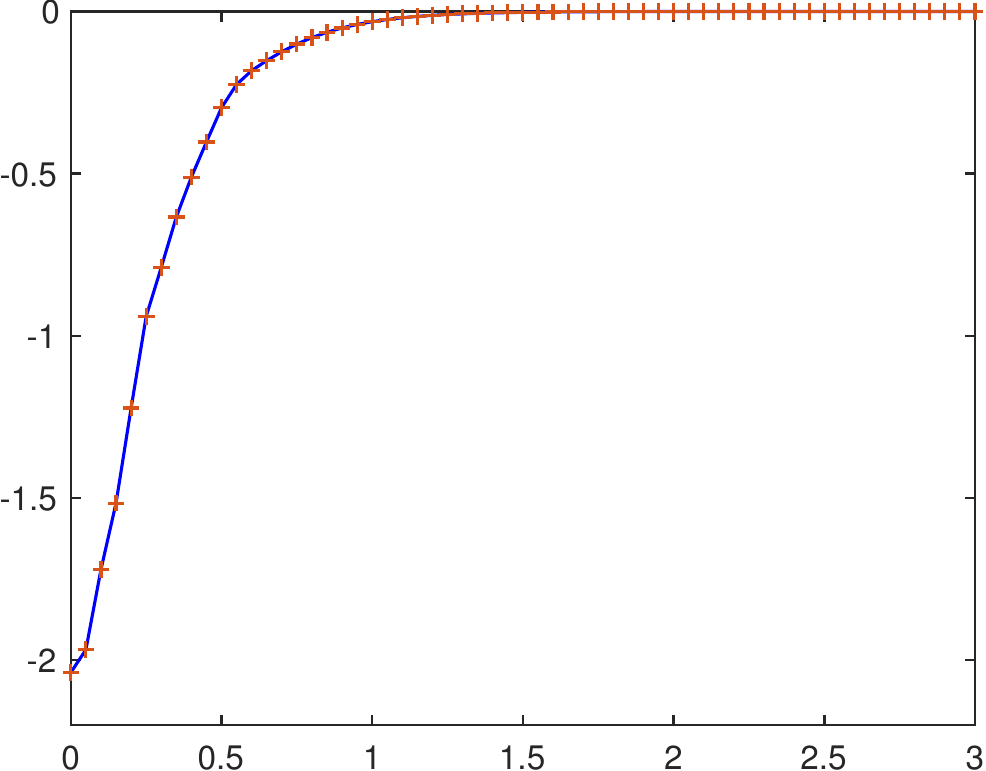}
\includegraphics[height=4truecm]{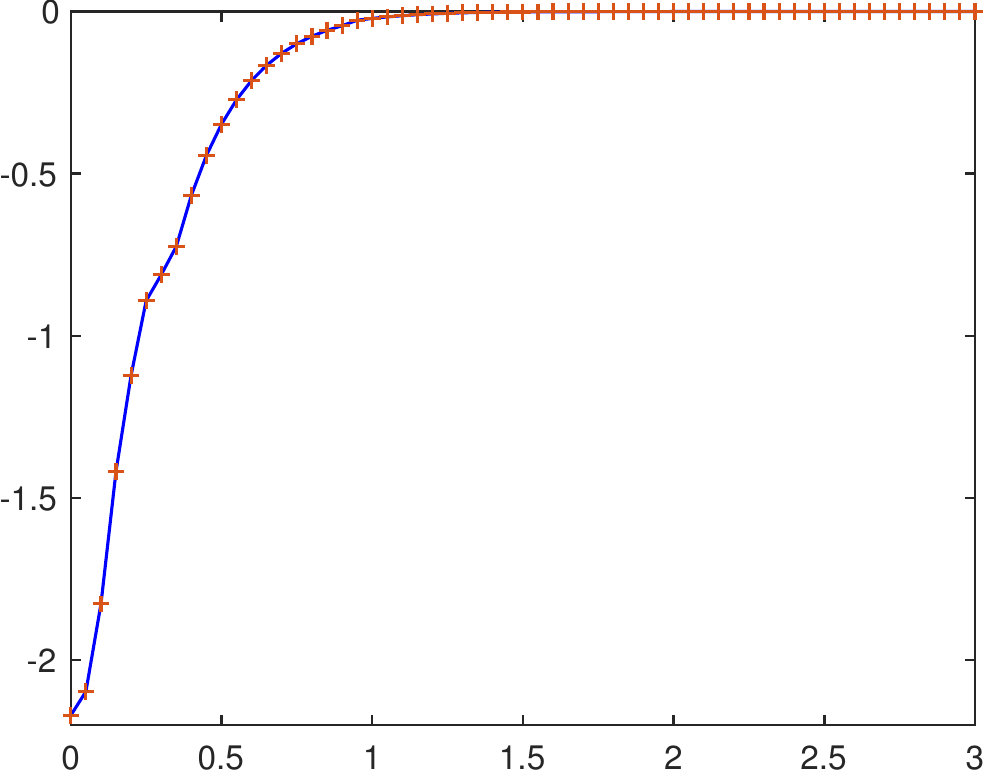}
\includegraphics[height=4truecm]{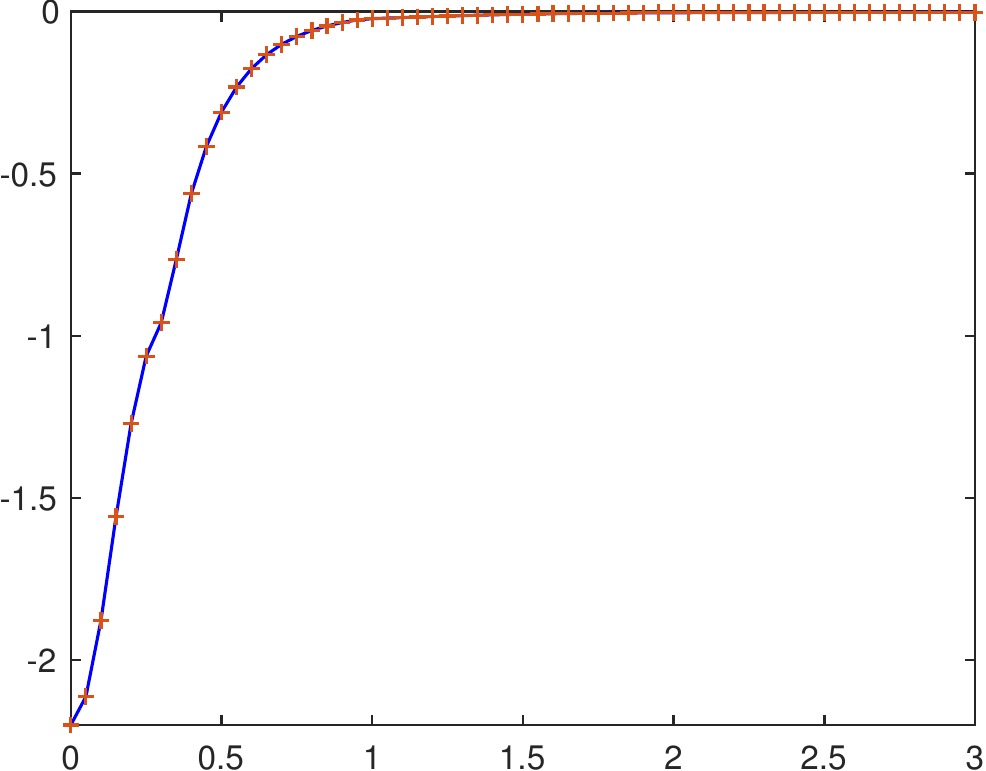}

\caption{\label{cv6} Test 2: Optimal HJB states computed with $r=4$ POD basis functions (top-left), difference between optimal solution with 4 and 2 POD basis functions (top middle), difference between optimal solution with 4 and 3 POD basis  functions (top-right). Optimal HJB controls with $r=4,3,2$ (bottom).}
\end{center}
\end{figure}

In Figure~\ref{cv6} we have represented on top the optimal solution (left) for $r=4$, the difference between optimal solution with 4 and 2 POD basis functions (top middle) and the difference between optimal solution with 4 and 3 POD basis functions (top right). On the bottom part we have represented the optimal controls for $r=2$, $3$ and $4$. Also in this case, the improvement with respect to the results in~\cite{Alla_Falcone_Volkwein} is remarkable. In particular, the optimal controls in Figure~\ref{cv6} compare very well with the
optimal LQR control of Figure~\ref{cvcontrol} even for the case with only $r=2$ basis functions in  our POD method. 

To conclude, in Figure~\ref{LQR_err} (left) we show the difference between the optimal LQR state and the optimal HJB state computed with $r=4$ POD basis functions. On the right, we show the relative errors $\left|u_{\rm HJB} - u_{LQR}\right| /\max(10^{-3}, \left| u_{LQR}\right|)$ of the optimal HJB controls with respect to the optimal LQR control for $r=2, 3$ and $4$. It can be seen a very good agreement between HJB and LQR optimal states. With respect to the optimal controls, we notice that whereas with $r=3$ and~$r=4$ POD basis functions the errors do not exceed 30\% and, indeed, they stay below 10\% most of the time, this is not the case of $r=2$ POD basis functions, where errors are above 100\% for more than half the time interval. However, let us observe that we are considering relative errors on the right of Figure~\ref{LQR_err} and that restricting ourselves to the time interval in which the optimal control is sufficiently away from zero the errors for $r=2$ are also below 35\%.  
\begin{figure}
\begin{center}
\includegraphics[height=3.9truecm]{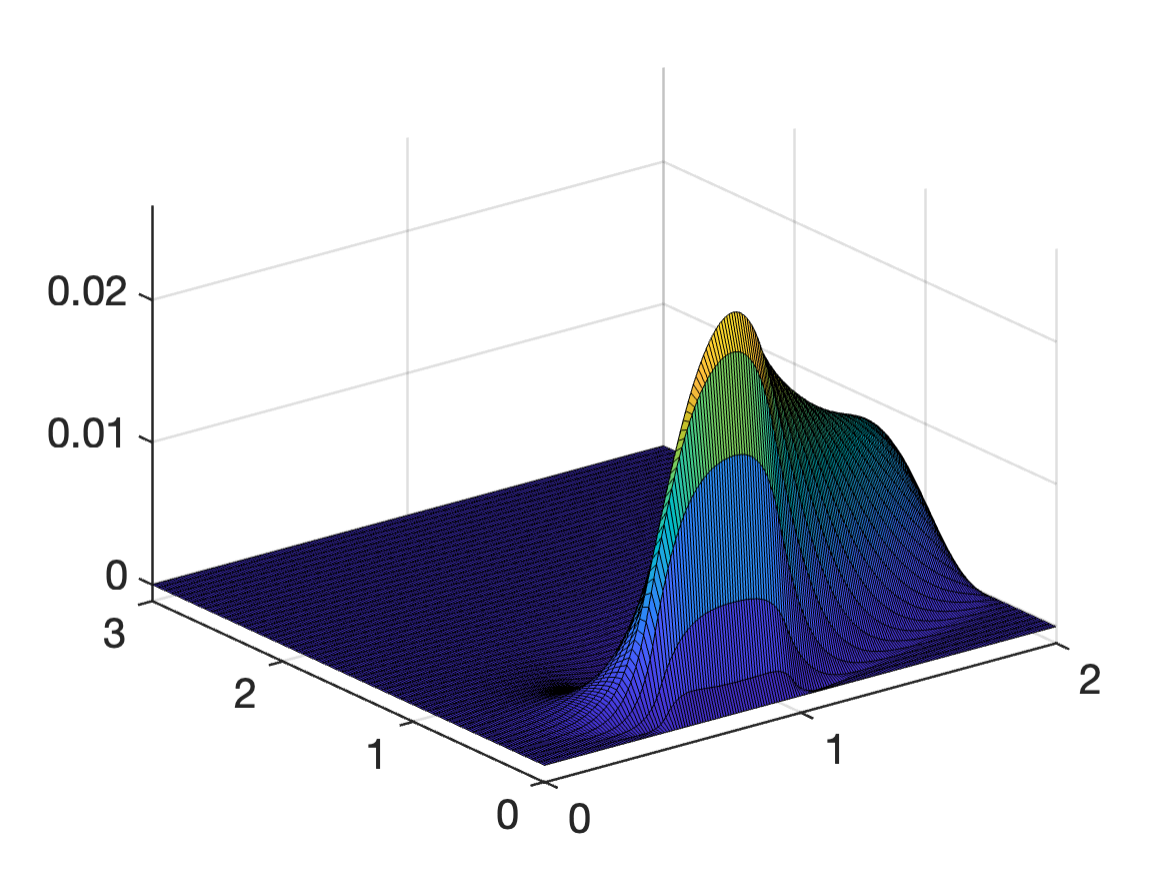}\quad
\includegraphics[height=3.9truecm]{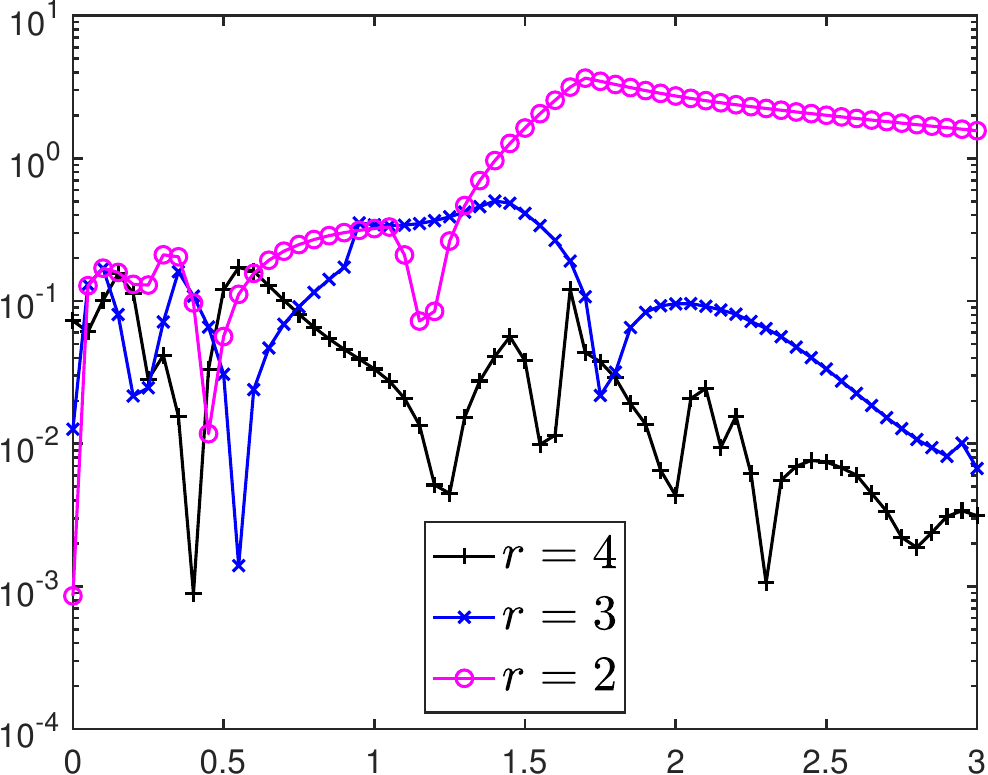}

\caption{\label{LQR_err} Test 2: difference between the optimal LQR state and the optimal HJB state computed with $r=4$ POD basis functions (left) and relative errors $\left|u_{\rm HJB} - u_{LQR}\right| /\max(10^{-3}, \left| u_{LQR}\right|)$  of the optimal HJB controls with respect to the optimal LQR control (right).
}
\end{center}
\end{figure}
Again, the results here (which correspond to $k_r=0.1$) compare favourably with those in the literature.
\begin{figure}[h]
\begin{center}
\includegraphics[height=3.9truecm]{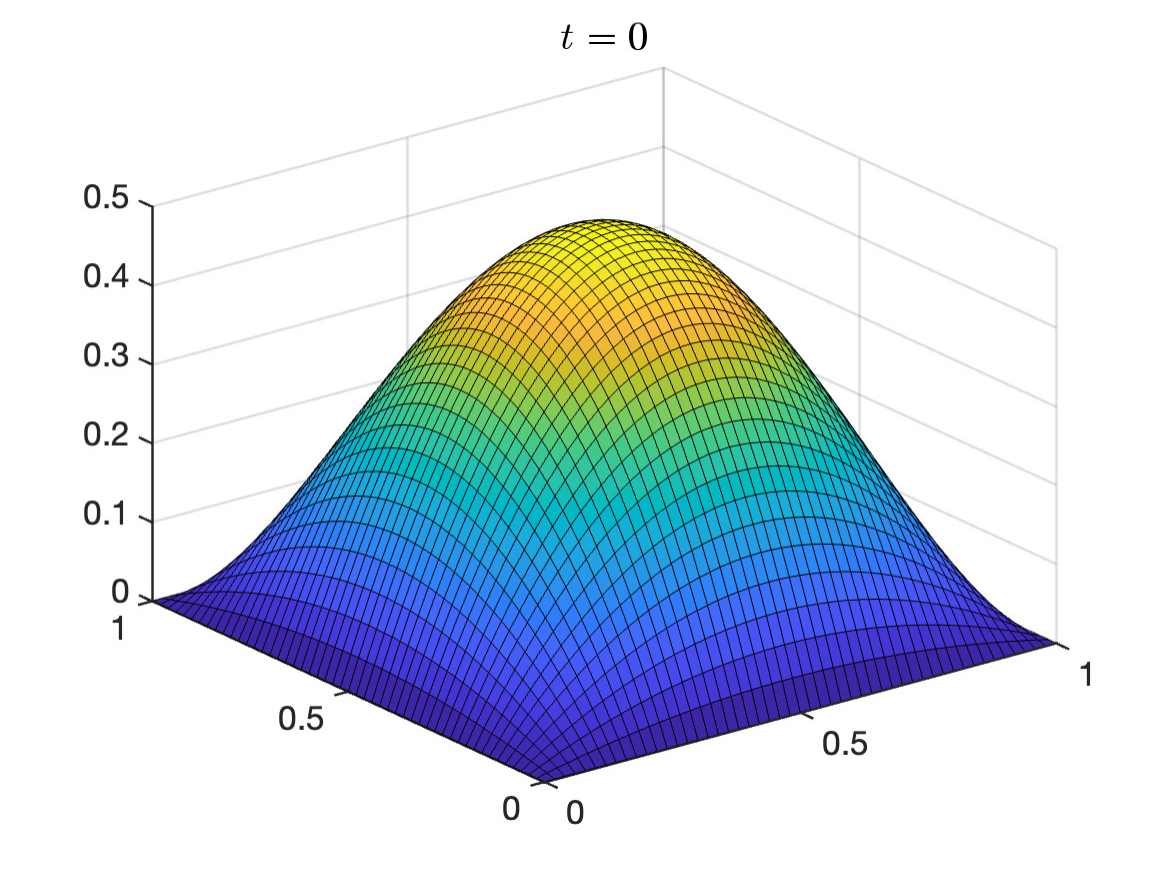}
\includegraphics[height=3.9truecm]{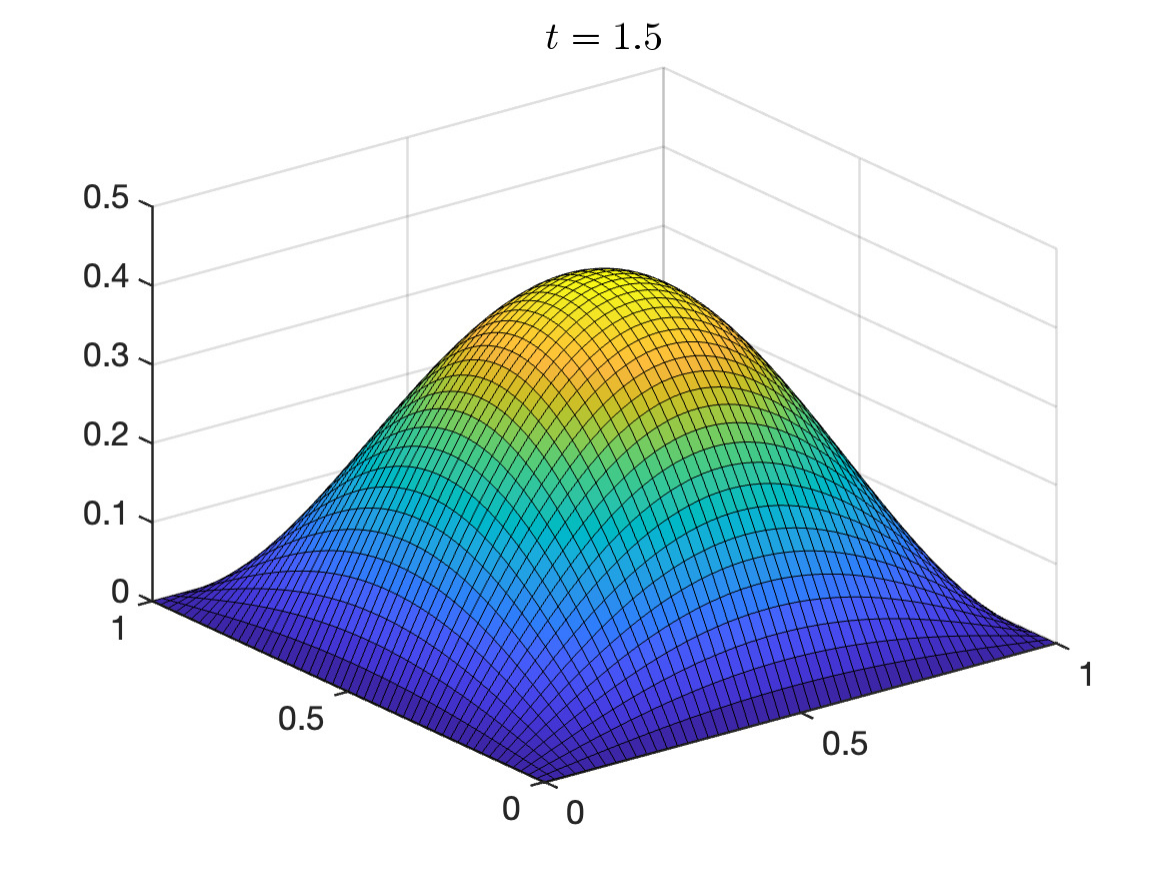}
\includegraphics[height=3.9truecm]{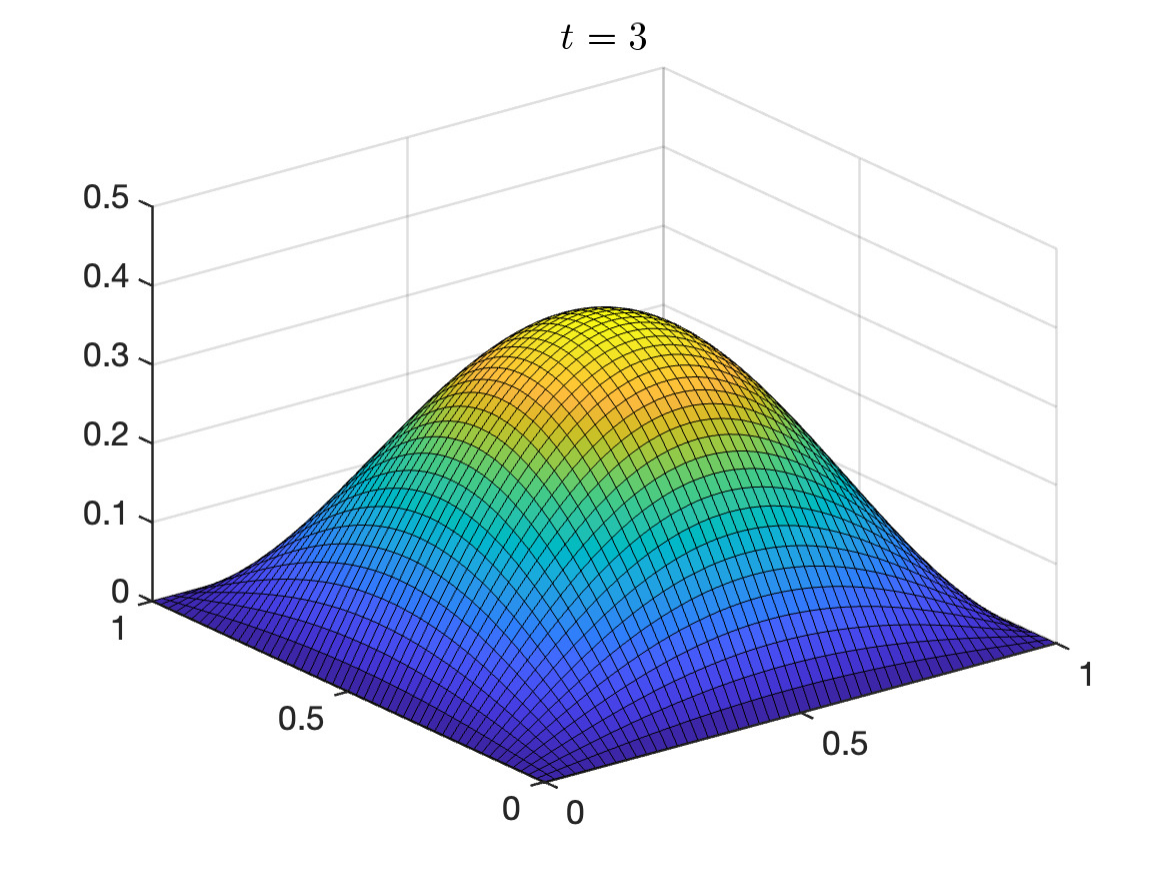}
\caption{\label{fig1_2D} Test 3: Uncontrolled solution at $t=0,1.5,3$.}
\end{center}
\end{figure}

\subsection{Test 3: A two-dimensional reaction-diffusion equation.}
\label{Test3}

We extend~\eqref{crd} to two dimensions. In particular, we consider,
\begin{equation}
\label{crd2D}
\begin{array}{rcll}
z_t-\varepsilon \Delta z  + (z^3-z) & =& ub\quad &\hbox{\rm in $\Omega\times (0,t_e)$},\\
z(\cdot,0)&=&z_0\quad &\hbox{\rm in $\Omega$},\\
z(\cdot,t) &=&0\quad&\hbox{\rm in~$\partial \Omega\times (0,t_e)$},
\end{array}
\end{equation}
with $\varepsilon=1/10$, $\Omega=[0,1]\times[0,1]$ and~$z: \Omega\times[0,t_e]$ denotes the state. The control~$u$ belongs to~${\mathbb U}_{\rm ad} = L^2(0,t_e,[u_a,u_b])$,
with $u_a=-1$ and~$u_a=1$. The cost function is as~\eqref{cost_f} but with the state measured in $L^2(\Omega)$ instead of~$L^2(I)$, that is
\begin{equation}
\label{cost_f}
\int_{0}^{t_e} e^{-\lambda t}\left(\left\| z(\cdot,t,u)\right\|_{L^2(\Omega)}^2 + \frac{1}{100}\left|u(t)\right|^2\right)\,dt,
\end{equation}
with $\lambda=1$ as before. Similarly to~Section~\ref{Test1}, we take $b(x,y)=z_0(x,y)=4x(1-x)y(1-y)$ and $t_e=3$. In Fig.~\ref{fig1_2D} we show the uncontrolled solution at
the initial time, at $t=t_e/2$ and~$t=t_e$.

For the finite-difference approximation, we consider $y:[0,t_e]\rightarrow {\mathbb R}^{(N-1)^2}$ with components $y_k(t)\approx z({\boldsymbol x}_k,t)$, where, for $k=(j-1)(N-1) +i$, $i,j=1,\ldots,N-1$,  ${\boldsymbol x}_k=
(x_i,y_j)$,  and $x_i=i\Delta x$, $y_j=j\Delta y$,  $\Delta x=\Delta y=1/N$, solution of
\begin{equation}
\label{fd_rd-2D}
\hat Cy_t=\frac{1}{10}\hat Ay + \hat C(\hat F(y) + u\hat B)
\end{equation}
where the components of~$\hat F$ and $\hat B$ are, respectively $\hat F_k = y_k(1-y_k^2)$,  $\hat B_k=4 x_i(1-x_i) y_j(1-y_j)$, $k=(j-1)(N-1) + i$, $i,j=1,\ldots,N-1$, and $\hat A$ and~$\hat C$  are $(N-1)^2\times (N-1)^2$  matrices given by $\hat A=I\otimes A + A\otimes I$, $\hat C= I\otimes C + C\otimes I$, where $I$ is the identity of order $N-1$, $\otimes$ represents the Kronecker product of matrices, and~$A$ and~$C$ are the matrices in~\eqref{matricesAyC}
so that, as in~Section~\ref{Test1}, the finite-difference discretization~\eqref{fd_rd-2D} is fourth-order convergent. The norm we consider in~${\mathbb R}^{(N-1)^2}$ is given by
\begin{equation}
\label{norma-2D}
\left\| y\right\|^2 = \Delta x\Delta y \sum_{k=1}^{(N-1)^2} y_k^2,
\end{equation}
so that it is a discrete version of the~$L^2$ norm in~$\Omega$.

In spite of the similarity with the one-dimensional case, more POD modes were needed to attain results similar to those in Section~\ref{Test1}, and for $r=5$ modes a set ${\mathbb U}_{\rm ad}$ with 81 controls uniformly distributed in~$[-1,1]$ were used for higher accuracy. The set~$\overline \Omega^r$ for $r=5$ was given by
$$
\overline \Omega^r =[-0.8,0.36]\times[-0.02,0.03]\times[-0.01,0.01]\times [-0.01,0.01]\times [-0.01,0.01].
$$
For this set we checked that condition \eqref{new_inv} holds.

On the top plots in Fig.~\ref{fig2_2D} (top three plots) we show the optimal HJB state at three different times  computed with $r=5$. Notice that the vertical scale is 5~times smaller that in the plots in~Fig.~\ref{fig1_2D}. We can see that the results are very similar to the one-dimensional case.
\begin{figure}[h]
\begin{center}
\includegraphics[height=3.9truecm]{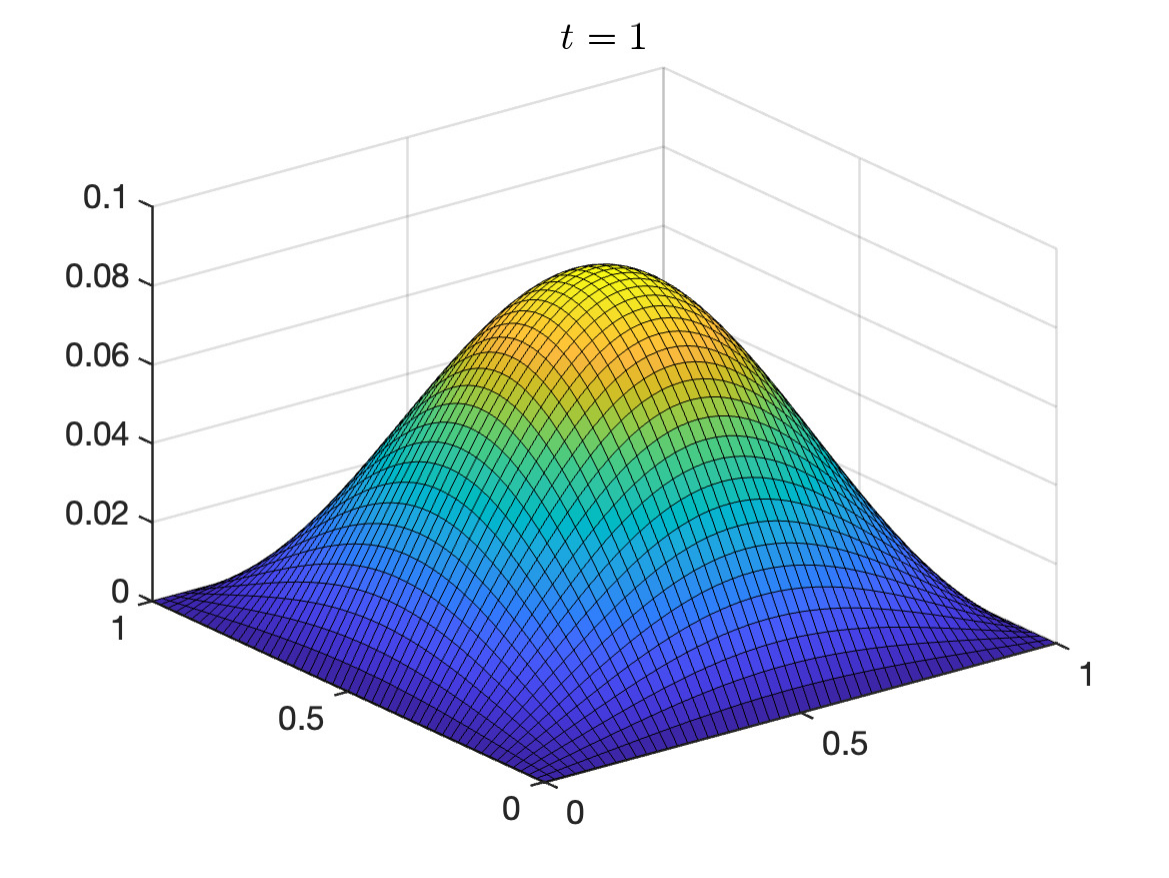}
\includegraphics[height=3.9truecm]{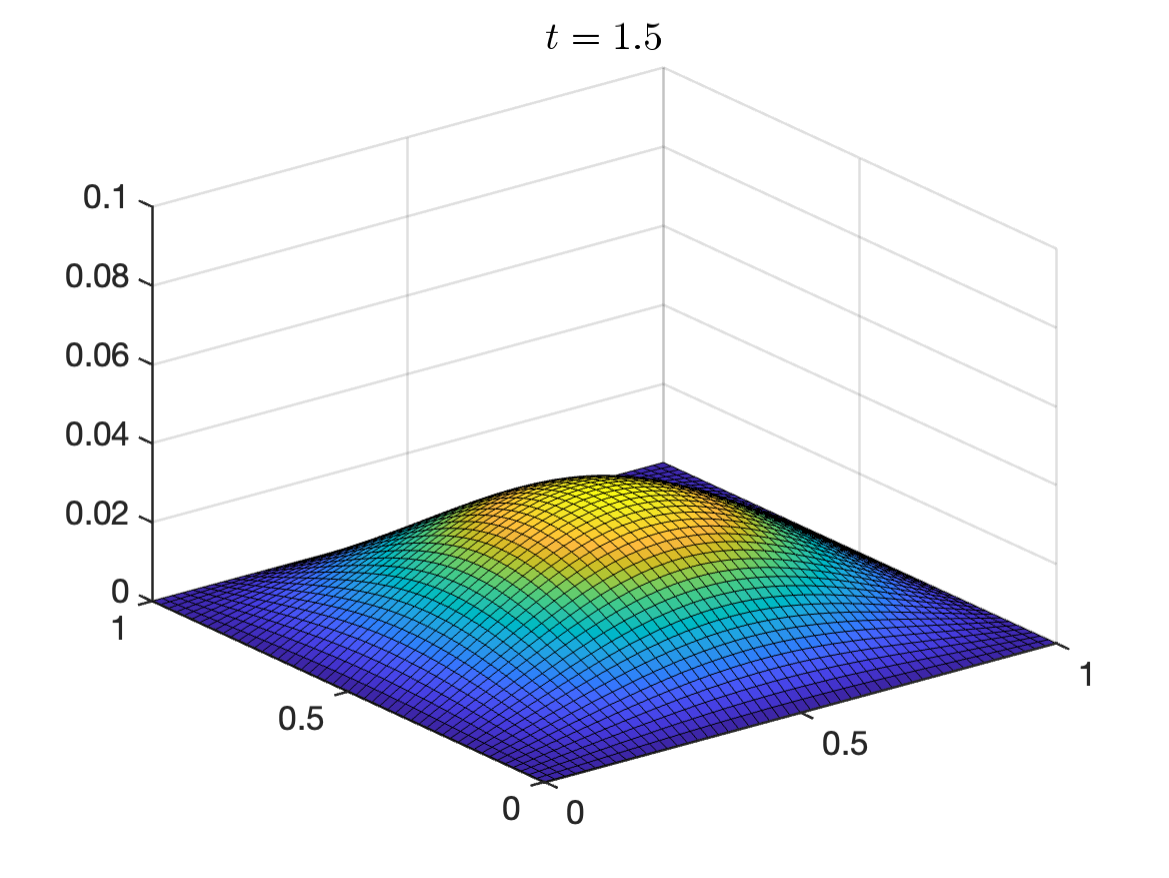}
\includegraphics[height=3.9truecm]{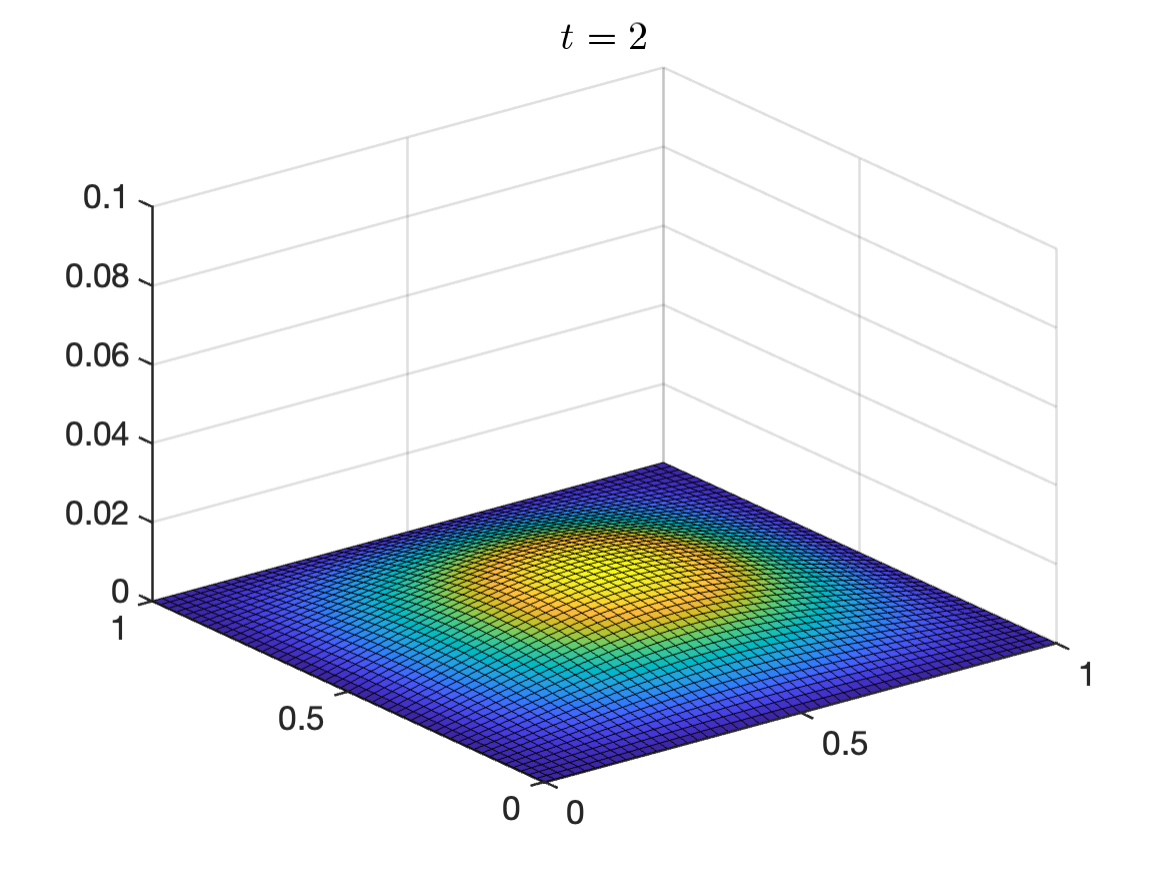}

\includegraphics[height=3.9truecm]{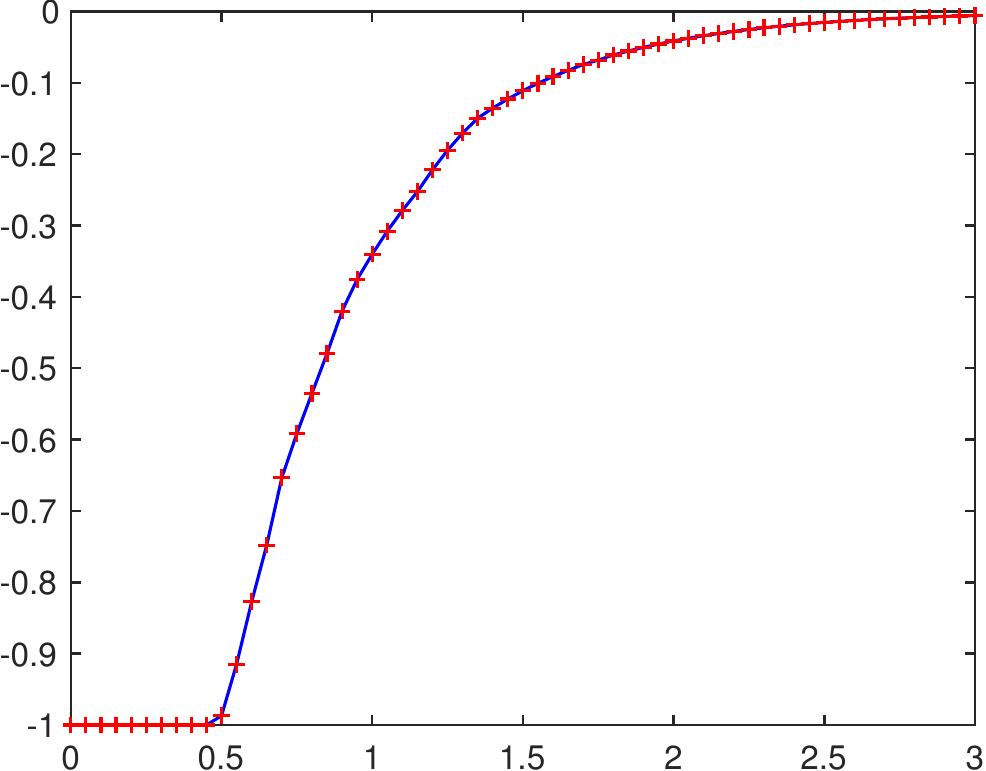}
\includegraphics[height=3.9truecm]{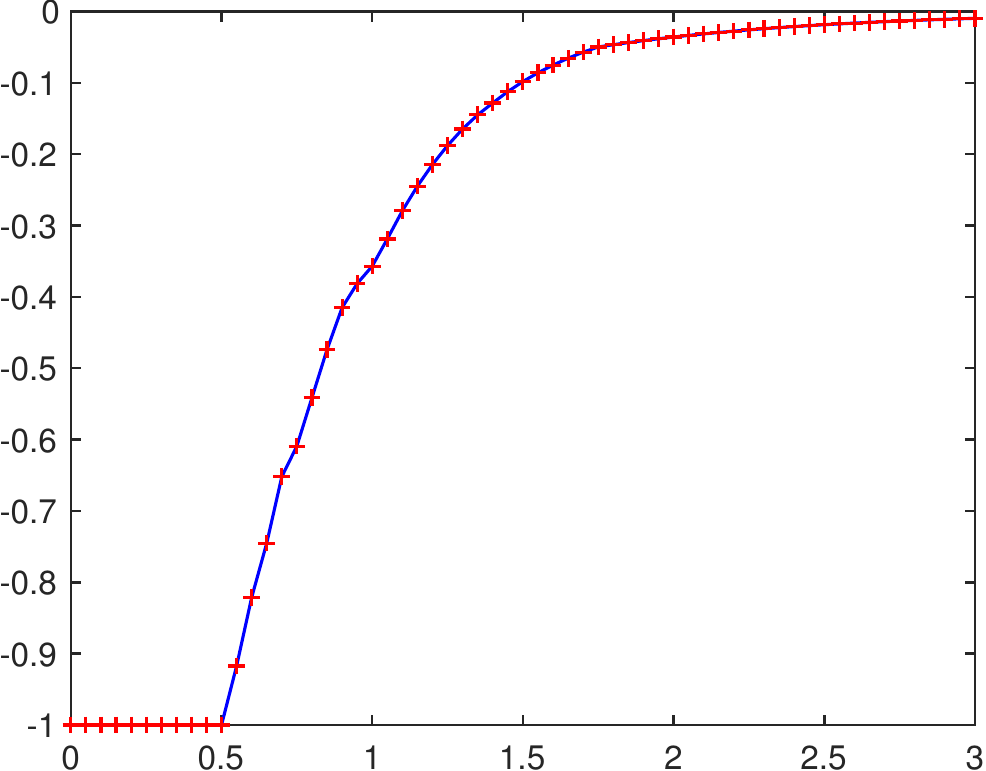}
\includegraphics[height=3.9truecm]{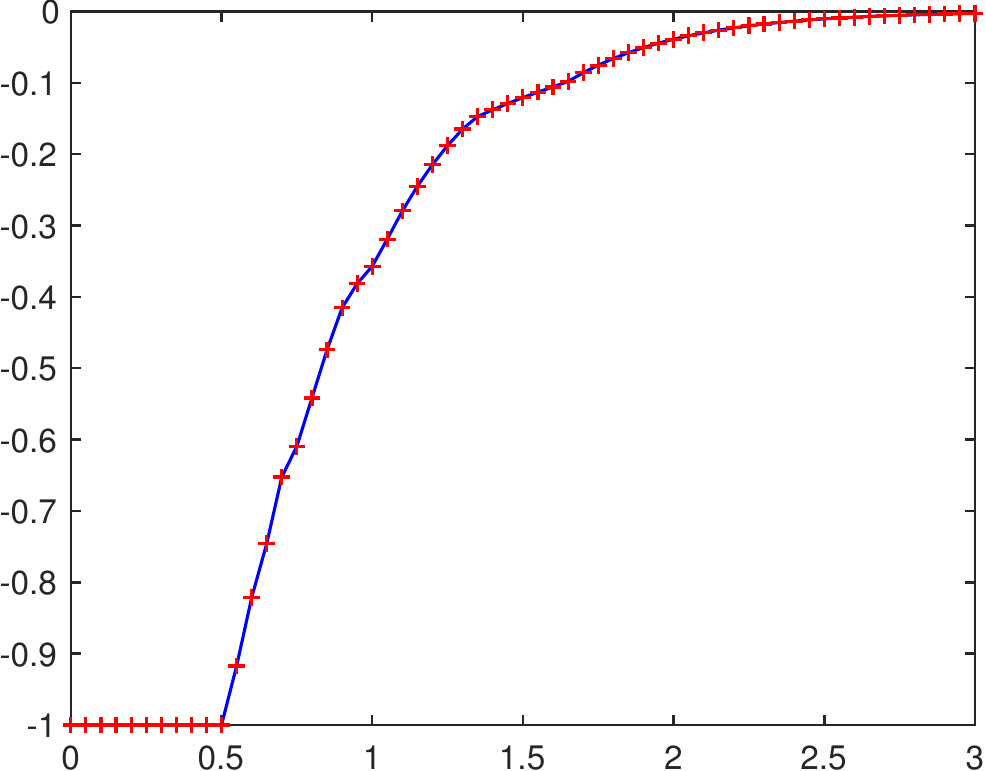}

\caption{\label{fig2_2D} Test 3: Optimal HJB states computed with r=5 POD basis functions at $t=1,1.5,2$ (top) and optimal Optimal HJB controls with $r=4,3,2$ (bottom). The red crosses correspond to the values of the controls that have been
joined by a blue line in the plot.}
\end{center}
\end{figure}

\section{Conclusions}
In this paper we introduce a reduced order method based on POD to mitigate the curse of dimensionality in the numerical approximation
of HJB equations. The novelty of the method is the use of snapshots based on temporal derivatives of the controlled nonlinear dynamical system. 

We carry out the error analysis of the method based on the recent results obtained in \cite{Javier_yo} that allow us to get sharper error bounds than those appearing in the literature. In particular, the factor $1/h$ where $h$ is the time step of the fully discrete method does not appear in our error bounds. Our error bounds are  optimal in terms of the time step $h$ and the mesh diameter of the reduced space $k_r$
and as usual depend also on the size of the tail of eigenvalues in the singular value decomposition. The use of snapshots based on time derivatives allow us to give a bound for some of the terms in the error that could not be bounded with the standard approach. 

\textcolor{red}{Numerous} numerical experiments are performed. \textcolor{red}{We check that the method behaves in practice as expected from the theoretical error analysis carried out it the present paper}.  \textcolor{red} {We show the importance of choosing a small tolerance for the fixed point iteration solving the POD fully discrete scheme \eqref{fully_discrete_pod}. We include a two-dimensional example to check the good performance of the method also in that case.} The new method we propose obtains better results than a similar POD method presented in~\cite{Alla_Falcone_Volkwein}. 
\textcolor{red}{Moreover, even for the method shown in~\cite{Alla_Falcone_Volkwein}, we have performed a numerical
experiment in which we also get better results.  This allows us to conclude the good performance, not
only of the new method introduced in this paper, but also of the method
presented in~\cite{Alla_Falcone_Volkwein}, when choosing accurate snapshots as well as taking denser enough sets  for the control variable
and, as mentioned above, small enough tolerance in the fixed point method solving \eqref{fully_discrete_pod}}.
\textcolor{red}{Finally, even in an example in which it is not possible to find an invariance set, we propose
a procedure that allows to apply the method proposed in this paper and that produce excellent results.}


\begin{thebibliography}{10}

\bibitem{Alla_Falcone_Volkwein} {\sc A. Alla, M. Falcone \& S. Volkwein},
\textit {Error analysis for POD approximations of infinite horizon problems via the dynamic programming approach}, SIAM J. Control Optim., 55 (2017),  3091--3115.

\bibitem{Alla_fal_sa} {\sc A. Alla, M. Falcone \& L. Saluzzi}, \textit{An efficient DP algorithm on a tree-structure for finite horizon optimal control problems}, SIAM J. Sci. Comput. 41 (2019) A2384?A2406.

\bibitem{Alla_et_all} {\sc A. Alla, H. Oliveira \& G Santin}, 
\textit{HJB-RBF Based Approach for the Control of PDEs}, J. Sci. Comput. 96 (2023), Paper 25, 27 pp.

\bibitem{Bardi}
{\sc M. Bardi \& I. Capuzzo-Dolcetta},
\textit {Optimal Control and Viscosity Solutions of Hamilton-Jacobi-Belmann Equations}, Springer Science+Business Media, LLC, New York, 1997.

\bibitem{boba_etal}
{\sc O. Bokanowski, N. Gammoudi \& H. Zidani},
 \textit {Optimistic planning algorithms for state-constrained optimal control}, Comput. Math. Appl. 109 (2022), 158--179.


\bibitem{boba2}
{\sc O. Bokanowski, J. Garcke, M. Griebel \& M. Klompmaker},
\textit {An adaptive sparse grid semi-Lagrangian scheme for first order Hamilton?Jacobi Bellman equations}, 
J. Sci. Comput. 55 (2013) 575--605.

\bibitem{locke-singler}
{\sc S.L.~Eskew \& J. R.~Singler},
\textit {A new approach to proper orthogonal decomposition with difference quotients},
 Adv. Comput. Math., 49(2), Paper No. 13, 33 (2023).


\bibitem{Kalise_et_al} {\sc D. Kalise \& K. Kunisch}, \textit{Polynomial approximation of high-dimensional Hamilton-Jacobi-Bellman equations and applications to feedback
control of semilinear parabolic PDEs}. SIAM J Sci Comput. 40 (2018), A629-A652.

\bibitem{Koc_et_al} {\sc B.~Koc, S.~Rubino, M.~Schneier, J.~Singler, \& T.~Iliescu}, \textit{On optimal pointwise in time error bounds and difference quotients
  for the proper orthogonal decomposition}, SIAM J. Numer. Anal. 59(4) (2021), 2163--2196.


\bibitem{Dolgov_et_al} {\sc S. Dolgov, D. Kalise \& L. Saluzzi}. \textit{Data-driven tensor  train gradient cross
approximation for Hamilton-Jacobi-Bellman equations}. SIAM J Sci
Comput. 45 (2023), A2153-A2184.

\bibitem{dante_etal} {\sc S. Dolgov, D. Kalise \& K. Kunisch}. \textit{Tensor decomposition methods for high-dimensional Hamilton-Jacobi-Bellman Equations}. SIAM J Sci
Comput. 43 (2021), A1625-A650.

\bibitem{eigel_et_al} {\sc M. Eigel, R. Schneider \& D. Sommer}, Dynamical low-rank approximations of solutions to the Hamilton-Jacobi-Bellman equation, Numer. Linear Algebra Appl., 30 (2023) Paper No. e2463, 20 pp.

\bibitem{falcone1} {\sc M. Falcone}, \textit{A numerical approach to the infinite horizon problem of deterministic control theory}, Appl. Math. Optim., 15 (1987), 1--13.

\bibitem{falcone11} {\sc M. Falcone}, \textit{Numerical solution of dynamic programming equations, in Optimal Control and Viscosity Solutions of Hamilton-Jacobi-Bellman Equations}, Birkh\"auser Boston, Boston, MA, 1997, pp. 471-504


\bibitem{falcone_dd} {\sc M. Falcone, P. Lacunara \& Seghini,} \textit{A splitting algorithm for Hamilton-Jacobi-Bellman equations}, Appl. Numer. Math., 15 (1994), 207--218.

\bibitem{Javier_yo} {\sc J. de Frutos \& J. Novo}, 
\textit{Optimal bounds for numerical approximations of infinite horizon problems based on dynamic programming approach}, SIAM. J. Control Optim. 61 (2023) 415--433.

\bibitem{Bosco_Volker_yo} {\sc B. Garc\'{\i}a-Archilla, V. John \& J. Novo}, \textit{POD-ROMs for incompressible flows including snapshots of the temporal derivative of the 
full order solution}, SIAM J. Numer. Anal. 61 (2023) 1340--1368.

\bibitem{Bosco_Volker_yo2} {\sc B. Garc\'{\i}a-Archilla, V. John \& J. Novo}, \textit{POD-ROM methods: from a finite set of snapshots to continuous-in-time approximations},  	arXiv:2403.06967 [math.NA].

\bibitem{Bosco_Julia_pointwise} {\sc B. Garc\'{\i}a-Archilla \& J. Novo}, \textit{Pointwise error bounds in POD methods without difference quotients},
arXiv:2407.17159 [math.NA].

\bibitem{Iliescu-Wang} {\sc T. Iliescu \& z. Wang}, \textit{Are the snapshot difference quotients needed in the proper orthogonal decomposition?} SIAM J. Sci.\ Comput.\ 36 (2014), A1221?A1250.

\bibitem{Ku-Vol} {\sc K. Kunisch \& S. Volkwein}, \textit{Galerkin proper orthogonal decomposition methods for parabolic problems}, Numer. Math., 90 (2001), 117-148.

\bibitem {Luo_et_al} {\sc B. Luo, W. Huai-Ning, T. Huang \& D. Liu} \textit{ Data-based approximate policy iteration for affine nonlinear continuous-time optimal control design}, Automatica, 50 (2014), 3281--3290.

\bibitem {maceneay1}{\sc W. M. McEneaney}, \textit{A curse of dimensionality free numerical method for solution of certain HJB PDEs.}, SIAM J. Control Optim. 46 (2007), 1239--1276.

\bibitem {maceneay2}{\sc W. M. McEneaney}, \textit{Convergence rate for a curse-of-dimensionality-free method for Hamilton-Jacobi-Bellman PDEs represented as maxima of quadratic forms}, SIAM J. Control Optim. 48 (2009), 2651--2685.


\bibitem {tonon} {\sc D. Tonon, M.S. Aronna \& D. Kalise}, Optimal control: novel directions and applications. Vol 1. Cham: Springer; 2017.

\bibitem{NDF}
{\sc L. F. Shampine \& M. W. Reichelt}, \textit{The {\sc Matlab} ODE Suite},
{SIAM J. Sci.\ Comput.\/}, {18} (1997), pp.~1--22.

  \end{thebibliography}
  \end{document}